\documentclass[11pt]{article}
\usepackage{amsfonts,epsf,amsmath,amssymb}
\usepackage{epsfig}
\usepackage{latexsym}
\usepackage{pgf,tikz}
\usepackage{bm}
\usepackage{lineno}
\usepackage{enumitem}

\newtheorem{theorem}{\bf Theorem}[section]
\newtheorem{corollary}[theorem]{\bf Corollary}

\newtheorem{proposition}[theorem]{\bf Proposition}

\newtheorem{remark}[theorem]{\bf Remark}
\newtheorem{definition}[theorem]{\bf Definition}

\newcommand{\qed}{\hfill $\square$ \bigskip}

\newcommand{\Fib}{{\mathcal F}}

\newcommand{\supp}{{\rm supp}}

\newcommand{\diam}{{\rm diam}}
\newcommand{\rad}{{\rm rad}}
\newcommand{\ecc}{{\rm ecc}}

\newcommand{\M}{\mathcal{M}}
\newcommand{\B}{{\mathcal B}}

\newcommand{\A}{{\mathcal A}}

\newcommand{\st}{ ~|~ }

\def\cp{\,\square\,}

\begin{document}
%\modulolinenumbers[5]
%\linenumbers
\title{Munarini graphs: a generalization of Fibonacci cubes and Pell graphs. Part I}

\author{
Michel Mollard\footnote{Institut Fourier, CNRS, Universit\'e Grenoble Alpes, France email: michel.mollard@univ-grenoble-alpes.fr}
}
\date{\today}
\maketitle

\begin{abstract}
The Fibonacci cube $\Gamma_n$ is the subgraph of the hypercube $Q_n$ induced by vertices with no consecutive $1$s. Munarini introduced Pell graphs, a variation of Fibonacci cubes defined on ternary strings.  A generalization of  Pell graphs to $(k+1)$-ary strings has recently been proposed. In this paper we introduce Munarini graphs, which constitute an alternative  generalization of Fibonacci cubes and  Pell graphs. One of the main advantages of Munarini graphs is that, unlike previously proposed generalization, they are daisy cubes, as are Fibonacci cubes and  Pell graphs. In this first article, we  study some of their fundamental properties including the size, the recursive structure,  the cube  and  maximal cube polynomials.

\end{abstract}

\noindent
{\bf AMS Subj. Class. }: 05C07,05C35

%%%%%%%%%%%%%%%%%%%%%%%%%%%%%%%%%%%%%%%%%%%%%%%%%%%%%%%%%%
\section{Introduction }

This work is dedicated to Emanuele Munarini, who passed away in 2024. Emanuele introduced Pell graphs~\cite{M-2019} and was also a co-author of the seminal paper on the structural and enumerative properties of Fibonacci cubes~\cite{MZ-2002}. It is therefore natural to name the generalization of these graphs introduced here after him.

Fibonacci strings are binary strings with no consecutive $1$s.    
The {\em Fibonacci cube} of dimension $n$, denoted $\Gamma_n$, is the subgraph of the hypercube $Q_n$  induced by Fibonacci strings of length $n$. 

Fibonacci cubes\ were introduced in 1993 by W.-J.~Hsu\ as a model for interconnection networks~\cite{H-1993a}. Since then, they have found numerous applications and have also proved to be of independent combinatorial interest. In 2013, a review article on Fibonacci cubes was written by 
Sandi Klav\v{z}ar~\cite{K-2013a} and recently a full book about them has been published~\cite{EKM-2023}.

The investigation of the properties of Fibonacci cubes has attracted many researchers and led to the development of a variety of interesting generalizations and variations, which are covered in an entire chapter of the book~\cite{EKM-2023}. Among these families of graphs are Lucas cubes, 
generalized Fibonacci cubes, Pell graphs, 
$k$-Fibonacci cubes, Fibonacci $p$-cubes and Fibonacci-run graphs.

The structure of Fibonacci cubes is  closely related to the Fibonacci numbers: $F_0 = 0$, $F_1=1$, $F_{n} = F_{n-1} + F_{n-2}$ for $n\geq 2$.
One of the generalizations of them is the $k$-Fibonacci sequence  defined for a positive integer $k$ as $F_{n,k} = k F_{n-1,k} + F_{n-2,k}$ for $n\geq 2$ with initial values $F_{0,k} = 0$ and $F_{1,k}=1$.
The  $1$-Fibonacci sequence is thus the Fibonacci sequence.
If $k=2$, we obtain another famous sequence, that of Pell numbers $P_{n}$, where $P_0 = 0$, $P_1=1$, $P_{n} = 2 P_{n-1} + P_{n-2}$ for $n\geq 2$.

Motivated by the Pell sequence, Emanuele Munarini~\cite{M-2019} introduced the Pell graph $\Pi_n$, defined on Pell strings, a subset of the set of strings of length $n$ over the alphabet \{0,1,2\}. The  structure of $\Pi_n$ reflects the recurrence $P_{n} = 2 P_{n-1} + P_{n-2}$.
Fibonacci cubes and Pell graphs~\cite[Theorem~9.68]{EKM-2023} belong to the class of daisy cubes~\cite{KM-2019a}, a family of isometric subgraphs of hypercubes.

Recently, inspired by the $k$-Fibonacci numbers for $k\geq2$, a generalization of Pell graphs, the generalized Pell graphs $\Pi _{n,k}$,  have been proposed by Vesna Ir\v{s}i\v{c}, Sandi Klav\v{z}ar and Elif Tan~\cite{IKT-2023}. Note that $\Pi _{n,k}$ is not a daisy cube for $k\geq3$ and $n\geq3$. 

Also inspired by $k$-Fibonacci numbers, Tomislav Do\v{s}li\v{c} and Luka Podrug~\cite{DP-2024} studied a second family of graphs, the metallic cubes $\Pi_{n}^{k}$, having the same order and size as $\Pi _{n,k}$. However, they do not constitute a  generalization of Pell graphs because $\Pi_{n}^{2}$ is not isomorphic to the Pell graph $\Pi_n$ for $n\geq 3$. Furthermore metallic cubes are not daisy cubes either.

In this article, we introduce a family of graphs that we call Munarini graphs $M_{n,k}$. These graphs have the same order and size as generalized Pell graphs, but, unlike them, they are daisy cubes. The inductive structure of Munarini graphs is another interpretation of the recurrence relation $F_{n,k} = k F_{n-1,k} + F_{n-2,k}$. Indeed, $M_{n,k}$ is constructed from $M_{n-2,k}$ and $k$ copies of $M_{n-1,k}$ connected along the star $S_{k-1}$, whereas in $\Pi_{n,k}$ and $\Pi_{n}^{k}$, the $k$ subgraphs of dimension $n-1$ form a path $P_{k}$ (see Figure~\ref{fig:Recursive}). Furthermore Munarini graphs generalize both Fibonacci cubes and Pell graphs since $M_{n,1}\cong\Gamma_{n-1}$ and $M_{n,2}\cong\Pi_{n}$. 

We will study the fundamental properties of Munarini graphs and show that several of their enumerative properties yield formulas very close  to those of Fibonacci cubes and Pell graphs.

This paper is organized as follows. In the next section we formally define the  concepts used in this paper and fix some notations. We then introduce Munarini graphs $M_{n,k}$ and determine their first properties, such as the number of edges and the recursive structure.  In Section~\ref{sec:sub} we study $M_{n,k}$ as subgraphs of hypercubes and Fibonacci cubes. In Section~\ref{sec:cub} we determine the cube polynomial and study hypercubes and maximal hypercubes in $M_{n,k}$.

	%%%%%%%%%%%%%%%%%%%%%%%%%%%%%%%%%%%%%%%%%%%%%%%%%%%%%%%%%%
%%%%%%%%%%%%%%%%%%%%%%%%%%%%%%%%%%%%%%%%%%%%%%%%%%%%%%%%%%
\section{Preliminaries}
\label{sec:basic}
%%%%%%%%%%%%%%%%%%%%%%%%%%%%%%%%%%%%%%%%%%%%%%%%%%%%%%%%%%
%%%%%%%%%%%%%%%%%%%%%%%%%%%%%%%%%%%%%%%%%%%%%%%%%%%%%%%%%%
We now introduce the concepts and notations needed in this paper.

%We denote by $[a,n]$ the set of integers $i$ such that $a\leq i \leq n$. 

Throughout this article, we use  the dotted union symbols $A\dot{\cup} B$ or $\dot{\bigcup}_{k\geq 0}{A_k}$ to enphasize that the sets are disjoint.

Let $k\geq 1$. A $(k+1)$-ary  string $s$ of length $l(s)$,  string for short,  is a sequence of $l(s)$ characters in $\{0,1,\ldots k\}$. Let $i\in\{0,1,\ldots k\}$, then $|s|_i$ is the number of occurrences of $i$ in $s$.
An alphabet is a set of symbols classically called letters; in our context, these letters may themselves be $(k+1)$-ary strings. For example $\Sigma=\{0,1,2\}$ or $\Sigma=\{0,1,22\}$ are alphabets over ternary strings. Let $\A$ be the set of words generated freely (as a monoid) by the alphabet $\Sigma$. This means that words $s \in \A$ can be written 
uniquely as a concatenation of zero or more letters from $\Sigma$. For example, for $\Sigma=\{0,1,22\}$, the string $0221$ belongs to $\A$, but $02221$ does not. Note that the empty string $\lambda$ always belongs to $\A$.
%Alternatively we say that the elements of $\A$ are strings over the alphabet $\Sigma$. If $s \in \A$ and  $a\in\Sigma$ then  $|s|_a$ is the number of use of $a$ in the word $s$ as an element of the monoid. For example for $\Sigma=\{0,1,22\}$ the string $0221$ belongs to  $\A$ but not  $02221$, $l(022221)=6$ and   $|022221|_{22}=2$.

Let $s$ be a $(k+1)$-ary  string and  $a\in\{0,\ldots,k\}$, then a \emph{run} of $a$s in $s$ is a maximal  substring of $s$ consisting entirely of the character $a$.
 
For a set of strings $\B$ and a string $s$ we denote by $s\B$ (respectively $\B s$) the sets of strings obtained by concatenation of $s$ with a string in $\B$ (respectively a string in $\B$ with $s$).

Let us consider the special case of binary strings so $k=1$, and let $B=\{0,1\}$. It will be convenient to identify elements $u = (u_1,\ldots, u_n)\in B^n$ and strings of length $n$ over $B$. We thus  briefly write $u$ as $u_1\ldots u_n$ and call $u_i$ the $i$-th coordinate of $u$. For $j\in\{1,\ldots,n\}$ the string  $u+\delta_j$ is defined as the string $v$ such that $v_j\neq u_j$ and $v_i=u_i$ for all $i\neq j$.
We use the power notation for concatenation of bits, for instance $0^n = 0\ldots 0\in B^n$. 

The Hamming distance between two strings in $B^n$ is the number of coordinates the strings differ.

The vertex set of $Q_n$, the \emph{hypercube of dimension $n$},  is the set $B^n$, two vertices being adjacent if and only if they differ in precisely one coordinate. We  say that an edge $uv$ of $Q_n$ uses the direction $i$ if $u$ and $v$ differ in the coordinate $i$, thus if $v=u+\delta_i$. 

The {\em Cartesian product}\ $G \cp H$ of graphs $G$ and $H$ has the vertex set $V(G) \times V(H)$, and $(g,h)$ is adjacent to $(g',h')$ if either $g = g'$ and $hh' \in E(H)$, or $gg' \in E(G)$ and  $h = h'$. Therefore, for $n\geq 1$,  $Q_n$ can be represented as the Cartesian product graph as follows: 
$$ Q_n=Q_{n-1}\cp K_2=Q_{n-1}\cp Q_1\,.$$

The \emph{star graph} on $n$ vertices $S_{n-1}$ is the complete graph $K_{1,n-1}$ that is, the  graph obtained by connecting a new vertex to  $n-1$ isolated vertices.

The {\em distance} $d_G(u,v)$\index{notations} between two vertices $u$ and $v$ of a connected graph $G$ is the 
number of edges on a shortest $u,v$-path. It is immediate that the distance between two vertices of $Q_n$ is the Hamming distance.
The {\em eccentricity} $\ecc_G(u)$  of $u\in V(G)$ is 
the maximum distance between $u$ and any other vertex of $G$. We may simply write 
$d(u,v)$ and $\ecc (u)$ when $G$ is clear from the context. 
The {\em diameter}\index{subject}  $\diam(G)$\index{notations} of a connected graph $G$ is the maximum distance between pairs  of vertices of $G$ or equivalently the maximum eccentricity $\ecc (x)$ among vertices $x$ of $G$. 
The {\em radius} $\rad(G) $ is the minimum  eccentricity. We note $\Delta(G)$ and $\delta(G)$ the maximum degree and minimum degree of $G$.

If $u$ and $v$ are vertices of a graph $G$, the \emph{interval} $I_G(u,v)$ between $u$ and $v$ (in $G$) is the set of vertices lying on shortest $u,v$-path, that is, $I_G(u,v) = \{w\st d(u,v) = d(u,w) + d(w,v)\}$. We also write $I(u,v)$ when $G$ is clear from the context.

Let $u$, $v$, and $w$ be different vertices of a connected graph $G$. 
A {\em median} of the triple $u,v,w$ 
is a vertex $x$ that simultaneously lies on a shortest $u,v$-path, a shortest $u,w$-path, and a shortest $v,w$-path.
The graph $G$ is a {\em median graph} if every 
triple of its vertices has a unique median. 

Hypercubes are median graphs and the unique median $m$ of $u,v,w$ is obtained  by the {\em majority rule}: the $i$th-coordinate of $m$ is the bit that appears at least twice among the $i$th-coordinates of  $u,v,w$.

A subgraph $G$ of a graph $H$ is {\em isometric} if $d_G(u,v) = d_H(u,v)$ holds for any $u,v\in V(G)$. 

Graphs isomorphic to isometric subgraphs of hypercubes are  called \emph{partial cubes}. In other words, a partial cube $G$ is a graph whose vertices can be labeled, for some $n$, with strings from $B^n$ such that the distance between any two vertices in $G$ is equal to the Hamming distance between their labels. Such labeling is called an isometric embedding in $Q_n$.
 
The {\em dimension} of a partial cube $G$ is the smallest integer
$n$ such that $G$ is an isometric subgraph of $Q_n$.
 
Many important classes of graphs are partial cubes, 
in particular trees, median graphs~\cite{M-1980}, benzenoid graphs, phenylenes, 
grid graphs and bipartite torus graphs.

If $G$ is a graph and $X\subseteq V(G)$, then $\langle X\rangle$ denotes the subgraph of $G$ induced by $X$.  
Let $\le$ be a partial order on $B^n$ defined by $u_1\ldots u_n \le v_1\ldots v_n$ if $u_i\le v_i$ holds for all $i\in \{1,\ldots,n\}$. For $X \subseteq B^n$ we define the {\em daisy cube generated by $X$}~\cite{KM-2019a} as the subgraph of $Q_n$  
$$Q_n(X) =\left\langle \{u\in B^n\st u\le x\ {\rm for\ some}\ x\in X \} \right\rangle\,.$$
%For $X \subseteq B^n$ we define the graph  $Q_n(X)$ as the subgraph of $Q_n$ with  
%$$Q_n(X) = \left\langle \{u\in B^n| u\le x\ {\rm for\ some}\ x\in X \} \right\rangle$$
%and say that $Q_n(X)$ is the {\em daisy cube generated by $X$}.
 Note that  if $\widehat{X}$ is the antichain consisting of the maximal elements of the poset $(X,\le)$, then $Q_n(\widehat{X}) = Q_n(X)$. We call the vertices of $Q_n(X)$ from  $\widehat{X}$ the {\em maximal vertices} of $Q_n(X)$. As noticed in the daisy cube introductory paper \cite{KM-2019a} we can alternatively say that $$Q_n(X) = \left\langle \bigcup_{x\in X} I_{Q_n}(x,0^n)\right\rangle=\left\langle \bigcup_{x\in \widehat{X}} I_{Q_n}(x,0^n)\right\rangle.$$

 Finally we will say that a graph $G$ is {\em a daisy cube} if there exists an isometric embedding  of $G$ in some hypercube $Q_n$ and a subset $X$ of $B^n$ such that $G$ is the daisy cube generated by $X$. Such an embedding will be called a {\em proper embedding}.

Daisy cubes are partial cubes~\cite{KM-2019a}. It is immediate that the class of daisy cubes is closed under the Cartesian product.

The following result is well-known (see~\cite{HIK-2011} for example).

\begin{proposition}\label{pro:bt}
In every induced subgraph $H$ of $Q_n$ isomorphic to $Q_k$ there exists a unique vertex of minimal Hamming weight, \emph{the bottom vertex} $b(H)$. There exists also a unique vertex of maximal Hamming weight, the \emph{top vertex} $t(H)$. 
Furthermore $b(H)$ and $t(H)$ are at distance $k$ and they characterize $H$ among the subgraphs of $Q_n$ isomorphic to $Q_k$. Note that $b(H) \le t(H)$. 
\end{proposition}

Let $G$ be a partial cube. If $H$ is an induced subgraph of $G$, it is also an induced subgraph of some $Q_n$. Thus Proposition~\ref{pro:bt} is still true for induced subgraphs of partial cubes.

The {\em support} of an induced hypercube $H$ of a partial cube  $G$, $\supp(H)$, is the set of coordinates that vary in $H$. Therefore, 
$$\supp(H) = \{i\in \{1,\ldots,n\}\st t_i=1, b_i=0\}\,.$$ 
$H$ is thus characterized by the couple $(t(H),\supp(H))$ or equivalently by the couple $(t(H),b(H))$. 
 \begin{remark}\label{re:daisy}
Let $G$ be a partial cube of dimension $n$. For $t\in V(G)$ and $b\leq t$  the hypercube of $Q_n$ defined by the couple $(t,b)$ is not necessarily a subgraph of $G$. However, if $G$ is a daisy cube, then this hypercube will be a subgraph of $G$ because $I_{Q_n}(b,t)\subset I_{Q_n}(0^n,t)\subset V(G)$. 
 \end{remark}

Let $(F_n)_{n\geq0}$ be the \emph{Fibonacci numbers}:
$F_0 = 0$, $F_1=1$, $F_{n} = F_{n-1} + F_{n-2}$ for $n\geq 2$.
A {\em Fibonacci string}  is a binary string without consecutive 1s. Let ${\cal F} $ be the set of Fibonacci strings and ${\cal F}_n$ those of length $n$.  Note that  ${\cal F} 0$  is the set of strings generated by the alphabet $\{0, 10\}$ and that ${\cal F}_{n-1} 0$ is the set of Fibonacci strings of length $n$ ending with $0$.

The {\em Fibonacci cube} $\Gamma_n$ ($n\geq 1$) is the subgraph of $Q_n$ induced by ${\cal F}_n$. Because of the empty string $\lambda$, $\Gamma_0 = K_1$. 

Since ${\cal F}_{n+2}=0{\cal F}_{n+1}\dot{\cup} 10{\cal F}_{n}$, ${\cal F}_{0}=\{\lambda\}$ and ${\cal F}_{1}=\{0,1\}$, by induction $|{\cal F}_{n}|=|V(\Gamma_{n})|=F_{n+2}$.  

The {\em Pell numbers} $(P_n)_{n\geq0}$ is another well-known sequence of integers defined by $P_0 = 0$, $P_1=1$, $P_{n} =  2 P_{n-1} + P_{n-2}$ for $n\geq 2$.

Motivated by this sequence, Emanuele Munarini~\cite{M-2019} introduced  Pell strings and  Pell graphs. 

A \emph{Pell string} is a  string over the alphabet $\{0,1,2\}$ without runs of 2s of odd length. Equivalently  Pell strings are elements of the monoid generated by $\{0,1,22\}$. Let ${\mathcal P}_n$ be the set of Pell strings of length $n$. In particular ${\mathcal P}_0=\{\lambda\}$ and ${\mathcal P}_1=\{0,1\}$. It is obvious by induction that the number of Pell strings of length $n$ is  $P_{n+1}$. 

For $n\geq 0$ Munarini defined the \emph{Pell graph} $\Pi_n$ as follows. Let $V(\Pi_n)={\mathcal P}_n$. Two vertices of $\Pi_n$ are adjacent if either one of them can be obtained from the other by replacing a 0 with a 1 (or vice versa), or by replacing a substring 11 with 22 (or vice versa) in such a way that the new string is again a Pell string. We have $\Pi_0=K_1$, $\Pi_1=K_2$.

The {\em $k$-Fibonacci sequence} which generalizes both Fibonacci and Pell numbers is defined, for a positive integer $k$, by $F_{0,k} = 0$, $F_{1,k}=1$, $F_{n,k} = kF_{n-1,k} + F_{n-2,k}$ for $n\geq 2$~\cite{FP-2007}.

From this definition, by the usual method, the generating function of the sequence $(F_{n,k})_{n\geq 0}$ is 
$$\sum_{n\geq 0}F_{n,k}t^n=\frac{t}{1-kt-t^2}.$$

Let $k\geq1$, then a \emph{$k$-generalized Pell string}, or generalized Pell string for short~\cite{IKT-2023}, is a string over the alphabet $\{0,1,\ldots, k\}$ without runs of $k$s of odd length. Equivalently  $k$-generalized Pell strings are the elements of the monoid generated by $\{0,1,\ldots, k-1,kk \}$. Let ${\cal F}_{n,k}$ be the set of $k$-generalized Pell strings of length $n$ .

Since ${\cal F}_{0,k}=\{\lambda\}$, ${\cal F}_{1,k}=\{0,1,\ldots,k-1\}$ and ${\cal F}_{n+2,k}= \dot{\bigcup}_{i= 0}^{k-1}{i{\cal F}_{n+1,k}}\dot{\cup} kk{\cal F}_{n,k}$  the number of strings in ${\cal F}_{n,k}$ is  $F_{n+1,k}$. 

Note that for $k=2$ the generalized Pell strings are the Pell strings.

The original definition in \cite{IKT-2023} assumes $k\geq2$, but the case $k = 1$ can also be considered. Since the monoid generated by $\{0, 11\}$ is isomorphic to that generated by $\{0, 10\}$, ${\cal F}_{n,1}$ is in bijection with the set of Fibonacci strings of length $n$ ending with a 0, itself in bijection with ${\cal F}_{n-1}$ for $n\geq1$. The bijection $\theta$ between ${\cal F}_{n,1}$ and ${\cal F}_{n-1}$ is therefore obtained by reading the strings  in ${\cal F}_{n,1}$ from left to right, keeping the 0s, changing the  $11$ to $10$, and then removing the final $0$ in the resulting string.

For $k\geq 2$ \emph{generalized Pell graphs} $\Pi_{n,k}$ are  defined as follows~\cite{IKT-2023}. The vertex set of $\Pi_{n,k}$ is ${\cal F}_{n,k}$. Two vertices of $\Pi_n$ are adjacent if either one of them can be obtained from the other by replacing a $i$ with a $i+1$ (or vice versa) where $i\in\{0,1,\ldots,k-2\}$ or by replacing a substring $(k-1)(k-1)$ with $kk$ (or vice versa) in such a way that the new string is again a generalized Pell string. Note that $\Pi_{n,2}=\Pi_n$, $\Pi_{0,k}=K_1$ and $\Pi_{1,k}$ is the path on $k$ vertices. The generalized Pell graph $\Pi_{2,3}$ is shown in Figure~\ref{fig:Pi23M23} (left) and the recursive structure of $\Pi_{n,k}$ in Figure~\ref{fig:Recursive} (left).

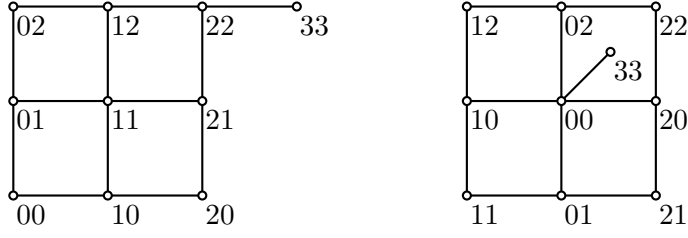
\begin{figure}[ht!]
\begin{center}
\begin{tikzpicture}[scale=0.5,style=thick]
\tikzstyle{every node}=[draw=none,fill=none]
\def\vr{3pt} % \vr = vertex radius;  Set \vr = 2/scale for uniform sizing of vertices

\begin{scope}[yshift = 0cm, xshift = 0cm]
\path (0,0) coordinate (v00);
\path (2.5,0) coordinate (v01);
\path (5,0) coordinate (v02);
\path (0,2.5) coordinate (v10);
\path (2.5,2.5) coordinate (v11);
\path (5,2.5) coordinate (v12);
\path (0,5) coordinate (v20);
\path (2.5,5) coordinate (v21);
\path (5,5) coordinate (v22);
\path (7.5,5) coordinate (v33);

%% edges %%
\draw (v00) -- (v10) -- (v20);
\draw (v01) -- (v11) -- (v21);
\draw (v02) -- (v12) -- (v22);

\draw (v00) -- (v01) -- (v02);\draw (v10) -- (v11) -- (v12);\draw (v20) -- (v21) -- (v22);
\draw (v22) -- (v33);

% 		\draw (x2) .. controls (2,2) and (2.5,1.5) .. (v2);
%% vertices %%%
\draw (v00)  [fill=white] circle (\vr);
\draw (v10)  [fill=white] circle (\vr);
\draw (v20)  [fill=white] circle (\vr);
\draw (v01)  [fill=white] circle (\vr);
\draw (v11)  [fill=white] circle (\vr);
\draw (v21)  [fill=white] circle (\vr);
\draw (v02)  [fill=white] circle (\vr);
\draw (v12)  [fill=white] circle (\vr);
\draw (v22)  [fill=white] circle (\vr);
\draw (v33)  [fill=white] circle (\vr);

\draw[right] (v00)++(-0.2,-0.5) node { $00$};\draw[right] (v10)++(-0.2,-0.5) node { $01$};
\draw[right] (v20)++(-0.2,-0.5) node {$02$};\draw[right] (v01)++(-0.2,-0.5) node { $10$};
\draw[right] (v11)++(-0.2,-0.5) node { $11$};\draw[right] (v21)++(-0.2,-0.5) node {$12$};
\draw[right] (v02)++(-0.2,-0.5) node { $20$};\draw[right] (v12)++(-0.2,-0.5) node { $21$};
\draw[right] (v22)++(-0.2,-0.5) node {$22$};\draw[right] (v33)++(-0.2,-0.5) node {$33$};

\end{scope}

\begin{scope}[yshift = 0 cm, xshift = 12cm]
%% vertices defined %%
\path (0,0) coordinate (v00);
\path (2.5,0) coordinate (v01);
\path (5,0) coordinate (v02);
\path (0,2.5) coordinate (v10);
\path (2.5,2.5) coordinate (v11);
\path (5,2.5) coordinate (v12);
\path (0,5) coordinate (v20);
\path (2.5,5) coordinate (v21);
\path (5,5) coordinate (v22);
\path (3.8,3.8) coordinate (v33);

%% edges %%
\draw (v00) -- (v10) -- (v20);
\draw (v01) -- (v11) -- (v21);
\draw (v02) -- (v12) -- (v22);

\draw (v00) -- (v01) -- (v02);\draw (v10) -- (v11) -- (v12);\draw (v20) -- (v21) -- (v22);
\draw (v11) -- (v33);

% 		\draw (x2) .. controls (2,2) and (2.5,1.5) .. (v2);
%% vertices %%%
\draw (v00)  [fill=white] circle (\vr);
\draw (v10)  [fill=white] circle (\vr);
\draw (v20)  [fill=white] circle (\vr);
\draw (v01)  [fill=white] circle (\vr);
\draw (v11)  [fill=white] circle (\vr);
\draw (v21)  [fill=white] circle (\vr);
\draw (v02)  [fill=white] circle (\vr);
\draw (v12)  [fill=white] circle (\vr);
\draw (v22)  [fill=white] circle (\vr);
\draw (v33)  [fill=white] circle (\vr);

\draw[right] (v00)++(-0.2,-0.5) node { $11$};\draw[right] (v10)++(-0.2,-0.5) node { $10$};
\draw[right] (v20)++(-0.2,-0.5) node {$12$};\draw[right] (v01)++(-0.2,-0.5) node { $01$};
\draw[right] (v11)++(-0.2,-0.5) node { $00$};\draw[right] (v21)++(-0.2,-0.5) node {$02$};
\draw[right] (v02)++(-0.2,-0.5) node { $21$};\draw[right] (v12)++(-0.2,-0.5) node { $20$};
\draw[right] (v22)++(-0.2,-0.5) node {$22$};\draw[right] (v33)++(-0.2,-0.5) node {$33$};

\end{scope}

\end{tikzpicture}
\end{center}
\caption{$\Pi_{2,3}\cong\Pi_2^3$(left) and $M_{2,3}$(right). \label{fig:Pi23M23}}
\end{figure}

\begin{figure}[ht]
\begin{center}
	\begin{tikzpicture}[scale=0.8,style=thick]
\tikzstyle{every node}=[draw=none,fill=none]
\def\vr{2pt} % \vr = vertex radius;  Set \vr = 2/scale for uniform 
\begin{scope}[yshift = 0cm, xshift = -3cm]

\foreach \i in {1,2,...,5}
{\path (0.1+0.5*\i,0.95) coordinate (u\i);}
\foreach \i in {1,2,...,5}
{\path (0.1+0.5*\i,1.35) coordinate (v\i);}
\foreach \i in {1,2,...,7}
{\path (0.1+0.5*\i,2.6) coordinate (a\i);}
\foreach \i in {1,2,...,7}
{\path (0.1+0.5*\i,2.9) coordinate (b\i);}
\foreach \i in {1,2,...,7}
{\path (0.1+0.5*\i,4.1) coordinate (c\i);}
\foreach \i in {1,2,...,7}
{\path (0.1+0.5*\i,4.65) coordinate (d\i);}
\foreach \i in {1,2,...,7}
{\path (0.1+0.5*\i,5.85) coordinate (e\i);}
\foreach \i in {1,2,...,7}
{\path (0.1+0.5*\i,6.4) coordinate (f\i);}
\draw (5,1.95) node {\footnotesize$k\hspace{-4pt}-\hspace{-4pt}1\Pi_{n-1,k}$};
\draw (5,3.45) node {\footnotesize$k\hspace{-4pt}-\hspace{-4pt}2\Pi_{n-1,k}$};
\draw (5,5.20) node {\footnotesize$i\Pi_{n-1,k}$};

%\draw (5,6.15) node {\footnotesize$1\Pi_{n-1,k}$};
\draw (5,6.95) node {\footnotesize$0\Pi_{n-1,k}$};
\draw (1.6,1.95) node {\footnotesize$k\hspace{-4pt}-\hspace{-4pt}1k\hspace{-4pt}-\hspace{-4pt}1\Pi_{n-2,k}$};
\draw (1.6,0.45) node {\footnotesize$kk\Pi_{n-2,k}$};
\foreach \i in {1,2,...,5}
{\draw (u\i) -- (v\i);}
\foreach \i in {1,2,...,7}
{\draw (a\i) -- (b\i);}
\foreach \i in {1,2,...,7}
{\draw[loosely dashed] (c\i) -- (d\i);}
\foreach \i in {1,2,...,7}
{\draw[loosely dashed] (e\i) -- (f\i);}

\foreach \i in {1,2}
{\draw (5,3.55+0.5*\i)  [fill=black] circle (\vr);}
\foreach \i in {1,2}
{\draw (5,5.30+0.5*\i)  [fill=black] circle (\vr);}

%\draw[rounded corners=12pt] (0,0) rectangle ++(5,1);
\draw[rounded corners=12pt] (0,1.5) rectangle ++(4,1);
\draw[rounded corners=12pt] (0,4.75) rectangle ++(4,1);
\draw[rounded corners=12pt] (0,3) rectangle ++(4,1);
\draw[rounded corners=12pt] (0,6.5) rectangle ++(4,1);
\draw[rounded corners=6pt] (0.2,0.2) rectangle ++(2.7,0.6);
\draw[rounded corners=6pt] (0.2,1.7) rectangle ++(2.7,0.6);
\end{scope}

\begin{scope}[yshift = 2.3cm, xshift = 9cm, scale=0.8]
\foreach \i in {1,2,...,4}
{\path (0.2+0.5*\i,-0.7) coordinate (u\i);}
\foreach \i in {1,2,...,4}
{\path (0.2+0.5*\i,0.2) coordinate (v\i);}

\foreach \i in {1,2,...,7}
{\path (-5.8+0.5*\i,3.8) coordinate (a\i);}
\foreach \i in {1,2,...,7}
{\path (-3.3+0.5*\i,2.3) coordinate (e\i);}

\foreach \i in {1,2,...,7}
{\path (3.3+0.5*\i,2.3) coordinate (f\i);}
\foreach \i in {1,2,...,7}
{\path (5.8+0.5*\i,3.8) coordinate (c\i);}

\foreach \i in {1,2,...,7}
{\path (0+0.5*\i,2.3) coordinate (b\i);}
\foreach \i in {1,2,...,7}
{\path (0+0.5*\i,3.8) coordinate (d\i);}

\foreach \i in {1,2,...,7}
{\draw (b\i) -- (d\i);}
\foreach \i in {1,2,...,7}
{\draw (a\i) -- (e\i);}
\foreach \i in {1,2,...,7}
{\draw (c\i) -- (f\i);}

\foreach \i in {1,2,...,5}
{\draw (u\i) -- (v\i);}
\draw (5.5,0.95) node {\footnotesize$0M_{n-1,k}$};
\draw (1.4,0.95) node {\footnotesize$00M_{n-2,k}$};
\draw (1.4,-1.55) node {\footnotesize$kkM_{n-2,k}$};
\draw (-4.1,5.8) node {\footnotesize$1M_{n-1,k}$};
\draw (8.5,5.8) node {\footnotesize $k\hspace{-4pt}-\hspace{-4pt}1M_{n-1,k}$};
\draw (2.0,5.8) node {\footnotesize$iM_{n-1,k}$};

\foreach \i in {1,2}
{\draw (-1.8+0.5*\i,5)  [fill=black] circle (\vr);}
\foreach \i in {1,2}
{\draw (4.4+0.5*\i,5)  [fill=black] circle (\vr);}

\draw[rounded corners=8pt] (0,0.5) rectangle ++(4,1);
\draw[rounded corners=8pt] (-6.2,4.5) rectangle ++(4,1);
\draw[rounded corners=8pt] (6.2,4.5) rectangle ++(4,1);
\draw[rounded corners=8pt] (0,4.5) rectangle ++(4,1);
\draw[rounded corners=5pt] (0.2,-1.8) rectangle ++(2.4,0.6);
\draw[rounded corners=5pt] (0.2,0.7) rectangle ++(2.4,0.6);
\end{scope}

\end{tikzpicture}
\end{center}
	\caption{Recursive decompositions of $\Pi_{n,k}$(left) and $M_{n,k}$(right).
	\label{fig:Recursive}}
\end{figure}
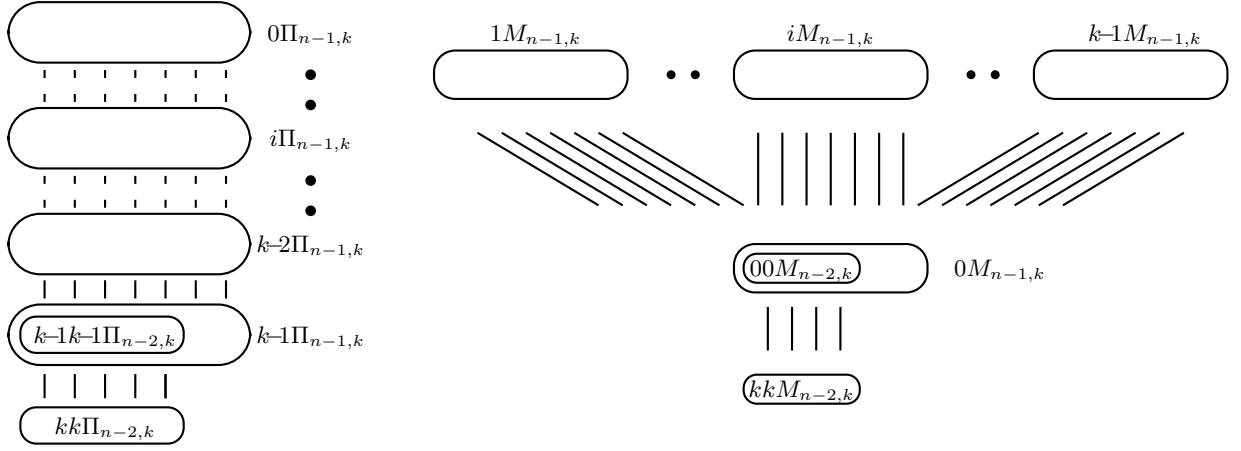

We can introduce now Munarini graphs.

\begin{definition}
For $k\geq 1$ and $n\geq0$ the \emph{Munarini graph} $M_{n,k}$ is the simple graph with vertex set ${\cal F}_{n,k}$ and edges defined as follows. Two vertices of $M_{n,k}$ are adjacent if either one of them can be obtained from the other by replacing a $0$ with a $i$ (or vice versa) where $i\in\{1,\ldots,k-1\}$ or by replacing a substring $00$ with $kk$ (or vice versa) in such a way that the new string is again a generalized Pell string. 
\end{definition}

\begin{figure}[ht!]
\begin{center}
\begin{tikzpicture}[scale=0.5,style=thick]
\tikzstyle{every node}=[draw=none,fill=none]
\def\vr{3pt} % \vr = vertex radius;  Set \vr = 2/scale for uniform sizing of vertices
\def\dv{1.5}
\def\dh{6}
\def\dz{0.5}

\begin{scope}[yshift = 2.5cm, xshift = 0cm]
\path (0,0) coordinate (v);
\draw (v)  [fill=white] circle (\vr);
\draw[right] (v)++(-0.2,-0.5) node {$\lambda$};
\end{scope}

\begin{scope}[yshift = 0cm, xshift = 6cm]
%% vertices defined %%
\path (0,0) coordinate (v1);
\path (0,2.5) coordinate (v0);
\path (0,5) coordinate (v2);

\draw (v1) -- (v0) -- (v2);

\draw (v1)  [fill=white] circle (\vr);
\draw (v0)  [fill=white] circle (\vr);
\draw (v2)  [fill=white] circle (\vr);

\draw[right] (v1)++(-0.2,-0.5) node { $1$};
\draw[right] (v0)++(-0.2,-0.5) node { $0$};\draw[right] (v2)++(-0.2,-0.5) node {$2$};

\end{scope}

\begin{scope}[yshift = 0 cm, xshift = 12cm]
%% vertices defined %%
\path (0,0) coordinate (v11);
\path (2.5,0) coordinate (v01);
\path (5,0) coordinate (v21);
\path (0,2.5) coordinate (v10);
\path (2.5,2.5) coordinate (v00);
\path (5,2.5) coordinate (v20);
\path (0,5) coordinate (v12);
\path (2.5,5) coordinate (v02);
\path (5,5) coordinate (v22);
\path (3.8,3.8) coordinate (v33);

%% edges %%
\draw (v11) -- (v10) -- (v12);
\draw (v01) -- (v00) -- (v02);
\draw (v21) -- (v20) -- (v22);

\draw (v11) -- (v01) -- (v21);\draw (v10) -- (v00) -- (v20);\draw (v12) -- (v02) -- (v22);
\draw (v00) -- (v33);

% 		\draw (x2) .. controls (2,2) and (2.5,1.5) .. (v2);
%% vertices %%%
\draw (v00)  [fill=white] circle (\vr);
\draw (v10)  [fill=white] circle (\vr);
\draw (v20)  [fill=white] circle (\vr);
\draw (v01)  [fill=white] circle (\vr);
\draw (v11)  [fill=white] circle (\vr);
\draw (v21)  [fill=white] circle (\vr);
\draw (v02)  [fill=white] circle (\vr);
\draw (v12)  [fill=white] circle (\vr);
\draw (v22)  [fill=white] circle (\vr);
\draw (v33)  [fill=white] circle (\vr);

\draw[right] (v11)++(-0.2,-0.5) node { $11$};\draw[right] (v10)++(-0.2,-0.5) node { $10$};
\draw[right] (v12)++(-0.2,-0.5) node {$12$};\draw[right] (v01)++(-0.2,-0.5) node { $01$};
\draw[right] (v00)++(-0.2,-0.5) node { $00$};\draw[right] (v02)++(-0.2,-0.5) node {$02$};
\draw[right] (v21)++(-0.2,-0.5) node { $21$};\draw[right] (v20)++(-0.2,-0.5) node { $20$};
\draw[right] (v22)++(-0.2,-0.5) node {$22$};\draw[right] (v33)++(-0.2,-0.5) node {$33$};
\end{scope}

\begin{scope}[yshift = -10 cm, xshift = 0cm]
%% vertices defined %%
\foreach \i in {0,1,...,2}
{\path (0+\i*\dh,0) coordinate (v\i11);
\path (2.5+\i*\dh,0+\dv) coordinate (v\i01);
\path (5+\i*\dh,0+2*\dv) coordinate (v\i21);
\path (0+\i*\dh,2.5) coordinate (v\i10);
\path (2.5+\i*\dh,2.5+\dv) coordinate (v\i00);
\path (5+\i*\dh,2.5+2*\dv) coordinate (v\i20);
\path (0+\i*\dh,5) coordinate (v\i12);
\path (2.5+\i*\dh,5+\dv) coordinate (v\i02);
\path (5+\i*\dh,5+2*\dv) coordinate (v\i22);
\path (3.8+\i*\dh,3.8+\dv+0.8) coordinate (v\i33);
}
\path (2.5+\dh-\dz,0+\dv+\dz) coordinate (v331);
\path (2.5+\dh-\dz,2.5+\dv+\dz) coordinate (v330);
\path (2.5+\dh-\dz,5+\dv+\dz) coordinate (v332);

%% edges %%
\foreach \i in {0,1,...,2}
{
\draw (v\i11) -- (v\i10) -- (v\i12);
\draw (v\i01) -- (v\i00) -- (v\i02);
\draw (v\i21) -- (v\i20) -- (v\i22);
\draw (v\i11) -- (v\i01) -- (v\i21);\draw (v\i10) -- (v\i00) -- (v\i20);\draw (v\i12) -- (v\i02) -- (v\i22);
\draw (v\i00) -- (v\i33);
}
\draw (v331) -- (v330) -- (v332);
\draw (v331) -- (v101); \draw (v330) -- (v100);\draw (v332) -- (v102);

% 		\draw (x2) .. controls (2,2) and (2.5,1.5) .. (v2);
%% vertices %%%
\foreach \i in {0,1,...,2}
{
\draw (v\i00)  [fill=white] circle (\vr);
\draw (v\i10)  [fill=white] circle (\vr);
\draw (v\i20)  [fill=white] circle (\vr);
\draw (v\i01)  [fill=white] circle (\vr);
\draw (v\i11)  [fill=white] circle (\vr);
\draw (v\i21)  [fill=white] circle (\vr);
\draw (v\i02)  [fill=white] circle (\vr);
\draw (v\i12)  [fill=white] circle (\vr);
\draw (v\i22)  [fill=white] circle (\vr);
\draw (v\i33)  [fill=white] circle (\vr);
}
\foreach \i in {0,1,...,2}
\foreach \j in {0,1,...,2}
{\draw (v0\i\j) -- (v1\i\j) -- (v2\i\j);
}
\draw (v033) -- (v133) -- (v233);
\draw (v330)  [fill=white] circle (\vr);\draw (v331)  [fill=white] circle (\vr);
\draw (v332)  [fill=white] circle (\vr);
\draw[left] (v332)++(0.1,+0.4) node { $332$};\draw[left] (v330)++(0.1,0) node { $330$};
\draw[left] (v331)++(0.1,0) node { $331$};\draw[right] (v100)++(-0.2,-0.5) node { $000$};
\draw[left] (v033)++(0.1,0) node { $133$};\draw[right] (v233)++(-0.3,0.5) node { $233$};
\draw[right] (v133)++(-0.4,0.7) node { $033$};

\draw[left] (v111)++(-0.0,-0.4) node { $011$};\draw[right] (v101)++(-0.2,-0.5) node { $001$};
\draw[left] (v110)++(-0.0,-0.4) node { $010$};\draw[left] (v112)++(-0.0,-0.4) node { $012$};
\draw[left] (v122)++(0.2,+0.5) node {$022$};\draw[left] (v121)++(0.3,+0.4) node {$021$};
\draw[right] (v120)++(-0.2,0.3) node {$020$};\draw[right] (v102)++(-0.5,0.8) node {$002$};

\draw[right] (v201)++(-0.2,-0.5) node { $201$};
\draw[right] (v221)++(-0.2,-0.5) node {$221$};\draw[right] (v220)++(-0.2,-0.5) node {$220$};
\draw[left] (v222)++(0.2,+0.5) node {$222$};\draw[right] (v211)++(-0.2,-0.5) node { $211$};
\draw[right] (v200)++(-0.2,-0.5) node { $200$};\draw[left] (v202)++(0.2,+0.5) node { $202$};
\draw[right] (v210)++(-0.2,-0.5) node { $210$};\draw[right] (v212)++(-0.2,-0.5) node { $212$};

\draw[left] (v000)++(0.2,+0.5) node { $100$};\draw[left] (v002)++(0.2,+0.5) node { $102$};
\draw[left] (v022)++(0.2,+0.5) node {$122$};\draw[left] (v011)++(0.2,+0.5) node { $111$};
\draw[left] (v010)++(0.2,+0.5) node { $110$};\draw[left] (v012)++(0.2,+0.5) node { $112$};
\draw[left] (v001)++(0.2,+0.5) node { $101$};\draw[left] (v021)++(0.2,+0.5) node { $121$};
\draw[left] (v020)++(0.2,+0.25) node { $120$};
%\draw[right] (v11)++(-0.2,-0.5) node { $11$};\draw[right] (v10)++(-0.2,-0.5) node { $10$};
%\draw[right] (v12)++(-0.2,-0.5) node {$12$};\draw[right] (v01)++(-0.2,-0.5) node { $01$};
%\draw[right] (v00)++(-0.2,-0.5) node { $00$};\draw[right] (v02)++(-0.2,-0.5) node {$02$};
%\draw[right] (v21)++(-0.2,-0.5) node { $21$};\draw[right] (v20)++(-0.2,-0.5) node { $20$};
%\draw[right] (v22)++(-0.2,-0.5) node {$22$};\draw[right] (v33)++(-0.2,-0.5) node {$33$};
\end{scope}
\end{tikzpicture}
\end{center}
\caption{Munarini graphs $M_{n,3}$ for $n\in \{0,1,2,3\}$. \label{fig:Muna}}
\end{figure}
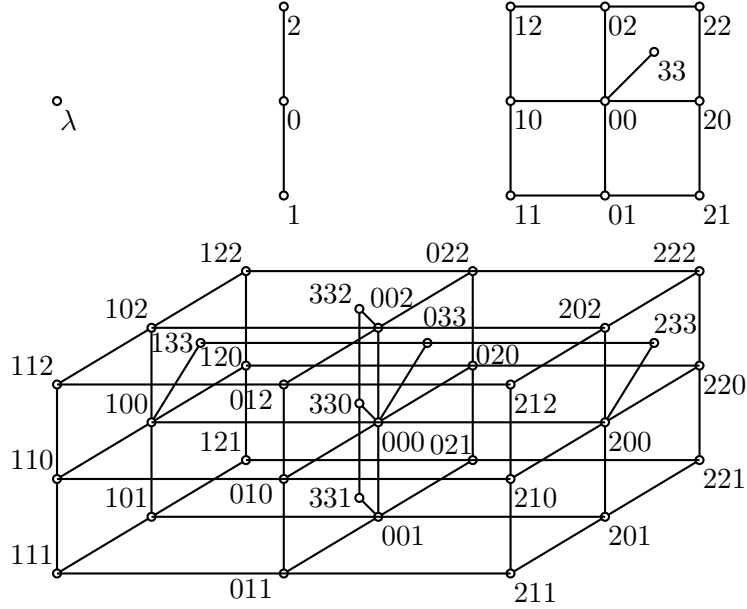
Note that $M_{0,k}=K_1$ and $M_{1,k}$ is the star on $k$ vertices $S_{k-1}$.The first $M_{n,3}$ are shown in Figure~\ref{fig:Muna}.

\begin{proposition}\label{prop:eq}
Let $n\geq 1$. Then  $M_{n,1}$ and $M_{n,2}$  are respectively isomorphic to $\Gamma_{n-1}$ and $\Pi_{n}$. 
\end{proposition}
\begin{proof}
Systematically replacing $1$ with $0$ and $0$ with $1$ in a Pell string produces a Pell string. Furthermore, this exchange between $0$ and $1$ transforms the adjacency definition in $M_{n,2}$ into that in $\Pi_n$, hence $M_{n,2}\cong\Pi_{n}$.

Similarly, the mapping $\theta$ between the vertices of $M_{n,1}$ and those of $\Gamma_{n-1}$ preserves adjacency. The two graph isomorphisms are illustrated in figure $\ref{fig:isom}$.
\end{proof}
\qed

\begin{figure}[ht!]
\begin{center}
\begin{tikzpicture}[scale=0.80,style=thick]
\tikzstyle{every node}=[draw=none,fill=none]
\def\vr{3pt} % \vr = vertex radius;  Set \vr = 2/scale for uniform sizing of vertices
\def\dec{7} % decalage
\begin{scope}[yshift = 0cm, xshift = 0cm]
%% vertices defined %%
\path (9.0-\dec,0) coordinate (u001);
\path (10.5-\dec,0) coordinate (u000);
\path (12-\dec,0) coordinate (u010);
\path (9.0-\dec,1.5) coordinate (u101);
\path (10.5-\dec,1.5) coordinate (u100);

\path (9.0,0) coordinate (v001);
\path (10.5,0) coordinate (v000);
\path (12,0) coordinate (v010);
\path (9.0,1.5) coordinate (v101);
\path (10.5,1.5) coordinate (v100);

\path (9.0+\dec,0) coordinate (w001);
\path (10.5+\dec,0) coordinate (w000);
\path (12+\dec,0) coordinate (w010);
\path (9.0+\dec,1.5) coordinate (w101);
\path (10.5+\dec,1.5) coordinate (w100);

%% edges %%
\draw (u010) -- (u000) -- (u001) -- (u101) -- (u100) -- (u000);
\draw (v010) -- (v000) -- (v001) -- (v101) -- (v100) -- (v000);
\draw (w010) -- (w000) -- (w001) -- (w101) -- (w100) -- (w000);

% 		\draw (x2) .. controls (2,2) and (2.5,1.5) .. (v2);
%% vertices %%%
\draw (u001)  [fill=white] circle (\vr);
\draw (u101)  [fill=white] circle (\vr);
\draw (u100)  [fill=white] circle (\vr);
\draw (u000)  [fill=white] circle (\vr);
\draw (u010)  [fill=white] circle (\vr);

\draw (v001)  [fill=white] circle (\vr);
\draw (v101)  [fill=white] circle (\vr);
\draw (v100)  [fill=white] circle (\vr);
\draw (v000)  [fill=white] circle (\vr);
\draw (v010)  [fill=white] circle (\vr);

\draw (w001)  [fill=white] circle (\vr);
\draw (w101)  [fill=white] circle (\vr);
\draw (w100)  [fill=white] circle (\vr);
\draw (w000)  [fill=white] circle (\vr);
\draw (w010)  [fill=white] circle (\vr);
%% text %%
\draw[above] (u001)++(0.5,0.0) node {$00\widehat{11}$};
\draw[above] (u101)++(0.4,0.0) node {$\widehat{11}\widehat{11}$};
\draw[above] (u100)++(0.4,0.0) node {$\widehat{11}00$};
\draw[above] (u000)++(0.5,0.0) node {$0000$};
\draw[above] (u010)++(0.4,0.0) node {$0\widehat{11}0$};

\draw[above] (v001)++(0.5,0.0) node {$00\widehat{10}$};
\draw[above] (v101)++(0.4,0.0) node {$\widehat{10}\widehat{10}$};
\draw[above] (v100)++(0.4,0.0) node {$\widehat{10}00$};
\draw[above] (v000)++(0.5,0.0) node {$0000$};
\draw[above] (v010)++(0.4,0.0) node {$0\widehat{10}0$};

\draw[above] (w001)++(0.4,0.0) node {$001$};
\draw[above] (w101)++(0.4,0.0) node {$101$};
\draw[above] (w100)++(0.4,0.0) node {$100$};
\draw[above] (w000)++(0.4,0.0) node {$000$};
\draw[above] (w010)++(0.4,0.0) node {$010$};

%%%
\draw (10.6-\dec,-0.7) node {$M_{4,1}$};
\draw (10.6,-0.7) node {$\Gamma_30$};
\draw (10.6+\dec,-0.7) node {$\Gamma_3$};

\draw (14-\dec,0.7) node {\huge{$\Rightarrow$}};
\draw (14,0.7) node {\huge{$\Rightarrow$}};

\end{scope}
\begin{scope}[yshift = -7cm, xshift = 5cm]
\def\deb{10}

\path (0,0) coordinate (a111);\path (0,3.6) coordinate (a000);
\path (1.2,2.4) coordinate (a001);\path (1.2,1.2) coordinate (a011);
\path (0,2.4) coordinate (a010);\path (0,1.2) coordinate (a101);
\path (-1.2,2.4) coordinate (a100);\path (-1.2,1.2) coordinate (a110);
\path (2.4,3.6) coordinate (a221);\path (1.2,4.8) coordinate (a220);
\path (-1.2,4.8) coordinate (a022);\path (-2.4,3.6) coordinate (a122);

\draw (a111) -- (a110) -- (a010) -- (a011) -- (a111);
\draw (a101) -- (a100) -- (a000) -- (a001) -- (a101);
\draw (a111) -- (a101);\draw (a110) -- (a100);
\draw (a010) -- (a000);\draw (a011) -- (a001);
\draw (a001) -- (a221) -- (a220) -- (a000);
\draw (a100) -- (a122) -- (a022) -- (a000);
\draw (a000)  [fill=white] circle (\vr);\draw (a100)  [fill=white] circle (\vr);
\draw (a001)  [fill=white] circle (\vr);\draw (a010)  [fill=white] circle (\vr);
\draw (a110)  [fill=white] circle (\vr);\draw (a101)  [fill=white] circle (\vr);
\draw (a111)  [fill=white] circle (\vr);\draw (a011)  [fill=white] circle (\vr);
\draw (a220)  [fill=white] circle (\vr);\draw (a221)  [fill=white] circle (\vr);
\draw (a022)  [fill=white] circle (\vr);\draw (a122)  [fill=white] circle (\vr);
\draw[above] (a000)++(0.0,0.13) node {$000$};
\draw[right] (a111)++(0.0,0.0) node {$111$};
\draw[left] (a110)++(0.0,0.0) node {$110$};
\draw[left] (a100)++(0.0,0.0) node {$100$};
\draw[left] (a122)++(0.0,0.0) node {$122$};
\draw[right] (a221)++(0.0,0.0) node {$221$};
\draw[right] (a001)++(0.0,0.0) node {$001$};
\draw[right] (a011)++(0.0,0.0) node {$011$};
\draw[left] (a101)++(0.0,0.0) node {$101$};
\draw[left] (a010)++(0.0,0.0) node {$010$};
\draw[above] (a022)++(0.0,0.0) node {$022$};
\draw[above] (a220)++(0.0,0.0) node {$220$};

\path (0+\deb,0) coordinate (b111);\path (0+\deb,3.6) coordinate (b000);
\path (1.2+\deb,2.4) coordinate (b001);\path (1.2+\deb,1.2) coordinate (b011);
\path (0+\deb,2.4) coordinate (b010);\path (0+\deb,1.2) coordinate (b101);
\path (-1.2+\deb,2.4) coordinate (b100);\path (-1.2+\deb,1.2) coordinate (b110);
\path (2.4+\deb,3.6) coordinate (b221);\path (1.2+\deb,4.8) coordinate (b220);
\path (-1.2+\deb,4.8) coordinate (b022);\path (-2.4+\deb,3.6) coordinate (b122);

\draw (b111) -- (b110) -- (b010) -- (b011) -- (b111);
\draw (b101) -- (b100) -- (b000) -- (b001) -- (b101);
\draw (b111) -- (b101);\draw (b110) -- (b100);
\draw (b010) -- (b000);\draw (b011) -- (b001);
\draw (b001) -- (b221) -- (b220) -- (b000);
\draw (b100) -- (b122) -- (b022) -- (b000);
\draw (b000)  [fill=white] circle (\vr);\draw (b100)  [fill=white] circle (\vr);
\draw (b001)  [fill=white] circle (\vr);\draw (b010)  [fill=white] circle (\vr);
\draw (b110)  [fill=white] circle (\vr);\draw (b101)  [fill=white] circle (\vr);
\draw (b111)  [fill=white] circle (\vr);\draw (b011)  [fill=white] circle (\vr);
\draw (b220)  [fill=white] circle (\vr);\draw (b221)  [fill=white] circle (\vr);
\draw (b022)  [fill=white] circle (\vr);\draw (b122)  [fill=white] circle (\vr);

\draw[above] (b000)++(0.0,0.13) node {$111$};
\draw[right] (b111)++(0.0,0.0) node {$000$};
\draw[left] (b110)++(0.0,0.0) node {$001$};
\draw[left] (b100)++(0.0,0.0) node {$011$};
\draw[left] (b122)++(0.0,0.0) node {$022$};
\draw[right] (b221)++(0.0,0.0) node {$220$};
\draw[right] (b001)++(0.0,0.0) node {$110$};
\draw[right] (b011)++(0.0,0.0) node {$100$};
\draw[left] (b101)++(0.0,0.0) node {$010$};
\draw[left] (b010)++(0.0,0.0) node {$101$};
\draw[above] (b022)++(0.0,0.0) node {$122$};
\draw[above] (b220)++(0.0,0.0) node {$221$};

\draw(5,1) node {\huge{$\Rightarrow$}};
\draw (0,-1) node {$M_{3,2}$};
\draw (0+\deb,-1) node {$\Pi_3$};

\end{scope}
\end{tikzpicture}
\end{center}
\caption{$M_{4,1}\cong\Gamma_{3}$ and $M_{3,2}\cong\Pi_{3}$
\label{fig:isom}}
\end{figure}
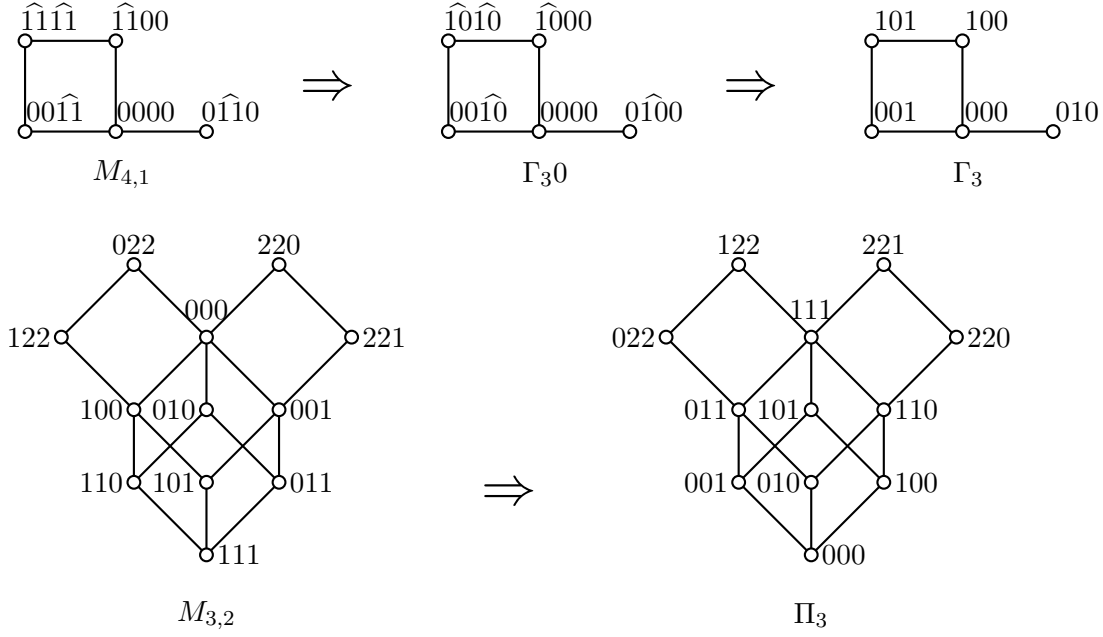

Since the vertex set  of $M_{n,k}$ is ${\cal F}_{n,k}$, the following proposition is immediate.
\begin{proposition}
The order of  $M_{n,k}$ is $F_{n+1,k}$ with generating function $$\sum_{n\geq0}{|V(M_{n,k})|t^n}=\frac{1}{1-kt-t^2}\,.$$
\end{proposition}

Let $n\geq2$. Since $V(M_{n,k})$  partitions as ${\cal F}_{n,k}= \dot{\bigcup}_{i= 0}^{k-1}{i{\cal F}_{n-1,k}}\dot{\cup} kk{\cal F}_{n-2,k}$ the Munarini graphs admit the following recursive decomposition (see Figure~\ref{fig:Recursive} right).

For $i\in\{0,1,\ldots,k-1\}$ each of the subgraphs induced by $i{\cal F}_{n-1,k}$ is isomorphic to $M_{n-1,k}$ and is denoted by $iM_{n-1,k}$. Similarly the subgraph induced by $kk{\cal F}_{n-2,k}$ is isomorphic to $M_{n-2,k}$ and is denoted by $kkM_{n-2,k}$.
According to the definition of Munarini graphs there are three kind of edges. First the edges of the subgraph induced by $\dot{\bigcup}_{i= 0}^{k-1}{i{\cal F}_{n-1,k}}$, subgraph isomorphic to the Cartesian product of $M_{n-1,k}$ with the star $S_{k-1}$.  The other edges are those of $kkM_{n-2,k}$ and the edges of the matching  between  $kkM_{n-2,k}$ and the subgraph $00M_{n-2,k}$ of $0M_{n-1,k}$. 

Therefore the size of $M_{n,k}$ satisfies 
\begin{align*}
|E(M_{n,k})|& =k|E(M_{n-1,k})|+(k-1)|V(M_{n-1,k})|+|E(M_{n-2,k})|+|V(M_{n-2,k})|\\
&= k|E(M_{n-1,k})|+|E(M_{n-2,k})|+ F_{n+1,k}-F_{n,k}\\
\text{with}\\
 |E(M_{0,k})|&=0 \text{ and } |E(M_{0,1})|=k-1.
\end{align*}

The size of generalized Pell graphs follows the same recurrence relation (see~\cite{IKT-2023}). Therefore $|E(M_{n,k})|=|E(\Pi_{n,k})|$. The generating function of $|E(\Pi_{n,k})|$ was determined from this recurrence relation in \cite{IKT-2023}. We have thus the following result.
\begin{proposition}
The generating function  of the number of edges in  $M_{n,k}$ is   $$\sum_{n\geq0}{|E(M_{n,k})|t^n}=\frac{(k-1)t+t^2}{(1-kt-t^2)^2}\,.$$
\end{proposition}

Munarini, Perelli Cippo and Zagaglia Salvi gave the following expression for the size of Fibonacci cubes as linear combination of $F_n$, $nF_n$ and $nF_{n+1}$
\begin{proposition}\cite{MPZ-2001}
For any positive integer $n$ the size of $\Gamma_n$ is 
$$|E(\Gamma_n)|=\frac{nF_{n+1}+2(n+1)F_n}{5}\,.
$$
\end{proposition}
Therefore $|E(M_{n,1})|=|E(\Gamma_{n-1})|=\frac{2nF_{n+1}-(n+1)F_n}{5}$. We can generalize this formula to all Munarini graphs.

\begin{proposition}\label{prop:sizelin}
For any positive integers $n$ and $k$ the sizes of $M_{n,k}$ and $\Pi_{n,k}$ are
$$|E(M_{n,k})|=|E(\Pi_{n,k})|=\frac{(k^2-k+2)nF_{n+1,k}+(k-2)(n+1)F_{n,k}}{k^2+4}\,.
$$
\end{proposition}
\begin{proof}
The generating functions of $F_{n,k}$, $nF_{n,k}$ and  $nF_{n+1,k}$ are, by differentiation,
\begin{align*}
f(t)=&\sum_{n\geq0}{F_{n,k}t^n}=\frac{t}{1-kt-t^2} \text{ and} \\ 
g(t)=&\sum_{n\geq0}{nF_{n,k}t^n}=\frac{t+t^3}{(1-kt-t^2)^2}\\
h(t)=&\sum_{n\geq0}{nF_{n+1,k}t^n}=\frac{kt+2t^2}{(1-kt-t^2)^2}\,.
\end{align*}

It is immediate to verify that 
\begin{equation}\label{eq:genesize}
\sum_{n\geq0}{|E(M_{n,k})|t^n}=\frac{(k-1)t+t^2}{(1-kt-t^2)^2}=af(t)+bg(t)+ch(t)
\end{equation}
with $a=b=\frac{k-2}{k^2+4}$ and $c=\frac{k^2-k+2}{k^2+4}$.
Since the coefficients of $t^n$ for any integer $n$ are equal in both sides of (\ref{eq:genesize}) the result follows.
\end{proof}
\qed

Note that, for Pell graphs, Proposition~\ref{prop:sizelin} reduces to $$|E(\Pi_{n})|=|E(M_{n,2})|=\frac{nF_{n+1,2}}{2}= \frac{n}{2}|V(\Pi_{n})|\,.$$

\section{Munarini graphs as subgraphs of hypercubes}
\label{sec:sub}
%%%%%%%%%%%%%%%%%%%%%%%%%%%%%%%%%%%%%%%%%%%%%%%%%%%%%%%%%%
%%%%%%%%%%%%%%%%%%%%%%%%%%%%%%%%%%%%%%%%%%%%%%%%%%%%%%%%%%
 It is useful to introduce the following sets of binary strings. 
Let  $A_0=0^k$ and, for $i\in \{1,\ldots,k\}$, let $A_i=0^{i-1}10^{k-i}$.

For given $k\geq 1$, a \emph{Munarini string} is a binary string generated by $$\{0^k,10^{k-1},010^{k-2},\ldots,0^{k-2}10,0^{k-1}10^k\}=\{A_0,A_1,\ldots,A_{k-1}, A_kA_0\}\,.$$

 Note that these strings are Fibonacci strings of length a multiple of $k$ and, for $n\geq0$, we denote by $\Fib_{n,k}^\star$ the set of Munarini strings of length $kn$.
Equivalently, Munarini strings are the elements of the monoid generated by $\A=\{A_0,A_1,\ldots,A_k\}$ with the constraint that every $A_k$ is followed by $A_0$.

Let $\Gamma_{n,k}^\star$ be the subgraph of $\Gamma_{kn}$ induced by $\Fib_{n,k}^\star$. 
For $k=1$, the Munarini strings are exactly the Fibonacci strings ending with a $0$, and $\Gamma_{n,1}^\star$ is $\Gamma_{n-1}0$, the subgraph of $\Gamma_{n}$ induced by $\Fib_{n-1}0$. Similarly, it is proved in~\cite{M-2019} that  $\Pi_n$ is isomorphic to a subgraph of $\Gamma_{2n-1}$. Using a generalization of the isomorphism between $M_{n,1}$ and $\Gamma_{n-1}0$ seen in Proposition~\ref{prop:eq} we will give a similar result for the general case of Munarini graphs.

\begin{theorem}\label{th:isoembed}
$M_{n,k}$ is isomorphic to $\Gamma_{n,k}^\star$, and in particular $M_{n,k}$ is a subgraph of $\Gamma_{kn-1}$.
\end{theorem}
\begin{proof}
Let $\Psi(i)=A_i$ for $i\in\{0,\ldots,k-1\}$ and $\Psi(kk)=A_kA_0$. The mapping $\Psi$ between the alphabets $\{0,1,\ldots,k-1,kk\}$ and $\{A_0,A_1,\ldots,A_{k-1},A_kA_0\}$  extends naturally to a bijection $\widehat{\Psi}$ between the monoids they generate. Since $\widehat{\Psi}$ maps a $(k+1)$-ary string of length $n$ to a binary string of length $kn$, $\widehat{\Psi}$ is therefore a bijection between $\Fib_{n,k}$ and $\Fib_{n,k}^\star$. 

Let $i,j\in\{0,1,\ldots,k\}$. Then, there exists an edge in $\Gamma_k$ between $A_i$ and $A_j$ if and only if $i=0$ and $j\in\{1,\ldots,k\}$ (or vice versa). 

Assume there exists an edge in $M_{n,k}$ between $u$ and $v$. Then two cases arise.

In the first case,  $u$ and $v$ can be written as  $u=a0b$ and $v=aib$ with $i\in\{1,\ldots,k-1\}$ where $a$ and $b$ are generalized Pell strings. We then have an edge in $\Gamma_{n,k}^\star$ between 
$\widehat{\Psi}(u)=\widehat{\Psi}(a)A_0\widehat{\Psi}(b)$ and $\widehat{\Psi}(v)=\widehat{\Psi}(a)A_i\widehat{\Psi}(b)$.

The second possibility is that $u=a00b$ and $v=akkb$. Then there exists an edge in $\Gamma_{n,k}^\star$ between 
$\widehat{\Psi}(u)=\widehat{\Psi}(a)A_kA_0\widehat{\Psi}(b)$ and $\widehat{\Psi}(v)=\widehat{\Psi}(a)A_0A_0\widehat{\Psi}(b)$. Note that no other adjacency is possible in $\Gamma_{n,k}^\star$ in this case, because for $i\neq0$, $\widehat{\Psi}(a)A_kA_i\widehat{\Psi}(b)$ does not belong to $\Fib_{n,k}^\star$. 
  
Conversely, assume there exists an edge in $\Gamma_{n,k}^\star$	between $xA_iy$ and $xA_0y$, $i\in\{1,\ldots,k-1\}$ , where $x,y$ are  Munarini strings. Then there exists an edge in $M_{n,k}$ between $\widehat{\Psi}^{-1}(xA_iy)=\widehat{\Psi}^{-1}(x)i\widehat{\Psi}^{-1}(y)$ and $\widehat{\Psi}^{-1}(xA_0y)=\widehat{\Psi}^{-1}(x)0\widehat{\Psi}^{-1}(y)$. The case of an edge between $xA_kA_0y$ and $xA_0A_0y$ is similar since $\widehat{\Psi}^{-1}(A_kA_0)=kk$ and $\widehat{\Psi}^{-1}(A_0A_0)=00$.
$\widehat{\Psi}$ is therefore a graph isomorphism between $M_{n,k}$ and $\Gamma_{n,k}^\star$ .

Since all strings in $\Fib_{n,k}^\star$ end with 0, removing this 0 makes $\Gamma_{n,k}^\star$ itself isomorphic to a subgraph of $\Gamma_{kn-1}$.
\end{proof}
\qed

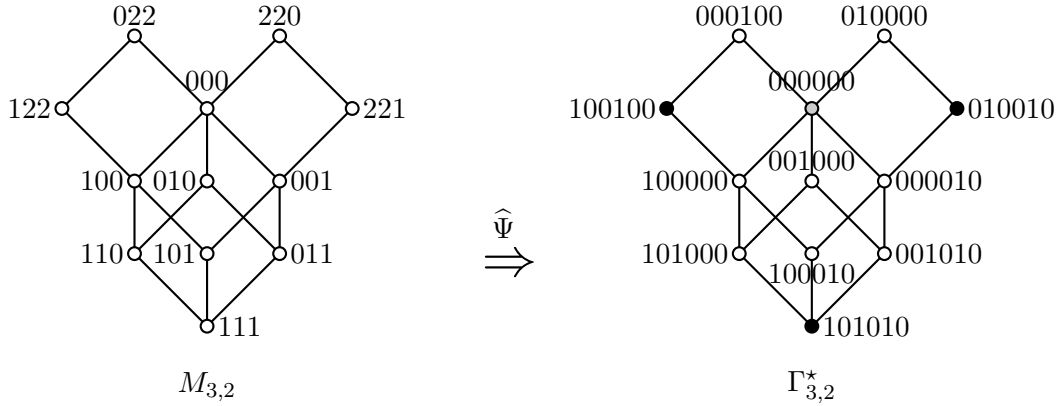
\begin{figure}[ht!]
\begin{center}
\begin{tikzpicture}[scale=0.80,style=thick]
\tikzstyle{every node}=[draw=none,fill=none]
\def\vr{3pt} % \vr = vertex radius;  Set \vr = 2/scale for uniform sizing of vertices
\def\dec{7} % decalage

\begin{scope}[yshift = 0cm, xshift = 0cm]
\def\deb{10}

\path (0,0) coordinate (a111);\path (0,3.6) coordinate (a000);
\path (1.2,2.4) coordinate (a001);\path (1.2,1.2) coordinate (a011);
\path (0,2.4) coordinate (a010);\path (0,1.2) coordinate (a101);
\path (-1.2,2.4) coordinate (a100);\path (-1.2,1.2) coordinate (a110);
\path (2.4,3.6) coordinate (a221);\path (1.2,4.8) coordinate (a220);
\path (-1.2,4.8) coordinate (a022);\path (-2.4,3.6) coordinate (a122);

\draw (a111) -- (a110) -- (a010) -- (a011) -- (a111);
\draw (a101) -- (a100) -- (a000) -- (a001) -- (a101);
\draw (a111) -- (a101);\draw (a110) -- (a100);
\draw (a010) -- (a000);\draw (a011) -- (a001);
\draw (a001) -- (a221) -- (a220) -- (a000);
\draw (a100) -- (a122) -- (a022) -- (a000);
\draw (a000)  [fill=white] circle (\vr);\draw (a100)  [fill=white] circle (\vr);
\draw (a001)  [fill=white] circle (\vr);\draw (a010)  [fill=white] circle (\vr);
\draw (a110)  [fill=white] circle (\vr);\draw (a101)  [fill=white] circle (\vr);
\draw (a111)  [fill=white] circle (\vr);\draw (a011)  [fill=white] circle (\vr);
\draw (a220)  [fill=white] circle (\vr);\draw (a221)  [fill=white] circle (\vr);
\draw (a022)  [fill=white] circle (\vr);\draw (a122)  [fill=white] circle (\vr);
\draw[above] (a000)++(0.0,0.13) node {$000$};
\draw[right] (a111)++(0.0,0.0) node {$111$};
\draw[left] (a110)++(0.0,0.0) node {$110$};
\draw[left] (a100)++(0.0,0.0) node {$100$};
\draw[left] (a122)++(0.0,0.0) node {$122$};
\draw[right] (a221)++(0.0,0.0) node {$221$};
\draw[right] (a001)++(0.0,0.0) node {$001$};
\draw[right] (a011)++(0.0,0.0) node {$011$};
\draw[left] (a101)++(0.0,0.0) node {$101$};
\draw[left] (a010)++(0.0,0.0) node {$010$};
\draw[above] (a022)++(0.0,0.0) node {$022$};
\draw[above] (a220)++(0.0,0.0) node {$220$};

\path (0+\deb,0) coordinate (b111);\path (0+\deb,3.6) coordinate (b000);
\path (1.2+\deb,2.4) coordinate (b001);\path (1.2+\deb,1.2) coordinate (b011);
\path (0+\deb,2.4) coordinate (b010);\path (0+\deb,1.2) coordinate (b101);
\path (-1.2+\deb,2.4) coordinate (b100);\path (-1.2+\deb,1.2) coordinate (b110);
\path (2.4+\deb,3.6) coordinate (b221);\path (1.2+\deb,4.8) coordinate (b220);
\path (-1.2+\deb,4.8) coordinate (b022);\path (-2.4+\deb,3.6) coordinate (b122);

\draw (b111) -- (b110) -- (b010) -- (b011) -- (b111);
\draw (b101) -- (b100) -- (b000) -- (b001) -- (b101);
\draw (b111) -- (b101);\draw (b110) -- (b100);
\draw (b010) -- (b000);\draw (b011) -- (b001);
\draw (b001) -- (b221) -- (b220) -- (b000);
\draw (b100) -- (b122) -- (b022) -- (b000);
\draw (b000)  [fill=black!20] circle (\vr);\draw (b100)  [fill=white] circle (\vr);
\draw (b001)  [fill=white] circle (\vr);\draw (b010)  [fill=white] circle (\vr);
\draw (b110)  [fill=white] circle (\vr);\draw (b101)  [fill=white] circle (\vr);
\draw (b111)  [fill=black] circle (\vr);\draw (b011)  [fill=white] circle (\vr);
\draw (b220)  [fill=white] circle (\vr);\draw (b221)  [fill=black] circle (\vr);
\draw (b022)  [fill=white] circle (\vr);\draw (b122)  [fill=black] circle (\vr);

\draw[above] (b000)++(0.0,0.13) node {$000000$};
\draw[right] (b111)++(0.0,0.0) node {$101010$};
\draw[left] (b110)++(0.0,0.0) node {$101000$};
\draw[left] (b100)++(0.0,0.0) node {$100000$};
\draw[left] (b122)++(0.0,0.0) node {$100100$};
\draw[right] (b221)++(0.0,0.0) node {$010010$};
\draw[right] (b001)++(0.0,0.0) node {$000010$};
\draw[right] (b011)++(0.0,0.0) node {$001010$};
\draw[below] (b101)++(0.0,-0.00) node {$100010$};
\draw[above] (b010)++(0.0,0.0) node {$001000$};
\draw[above] (b022)++(0.0,0.0) node {$000100$};
\draw[above] (b220)++(0.0,0.0) node {$010000$};

\draw(5,1) node {\huge{$\Rightarrow$}};
\draw(4.9,1.7) node {$\widehat{\Psi}$};
\draw (0,-1) node {$M_{3,2}$};
\draw (0+\deb,-1) node {$\Gamma_{3,2}^{\star}$};

\end{scope}
\end{tikzpicture}
\end{center}
\caption{$\widehat{\Psi}(M_{3,2})=\Gamma_{3,2}^{\star}$, in black vertices in $X$ of $\Gamma_{3,2}^{\star}$ as the daisy cube $Q_{6(X)}$.
\label{fig:isom2}}
\end{figure}

Fibonacci cubes~\cite[Theorem~10]{K-2005}, Pell graphs~\cite{M-2019} and generalized Pell graphs~\cite{IKT-2023} are median graphs.
Munarini's proof that Pell graphs are median graphs can be extended to Munarini graphs as a consequence of Theorem~\ref{th:isoembed}. 
\begin{theorem}
If $n\geq1$ and $k\geq1$, then $M_{n,k}$ is a median graph.
\end{theorem}
\begin{proof}
By a result by Mulder~\cite{M-1978}, an induced subgraph $G$ of an hypercube, such that, for any three vertices of $G $ their median in the hypercube belongs to $G$, is a median graph. Let us then consider three vertices $u,v,w$ of $\Gamma_{n,k}^\star$. Write $u=A_{u_1}A_{u_2}\ldots A_{u_n}$, $v=A_{v_1}A_{v_2}\ldots A_{v _n}$ and $w=A_{w_1}A_{w_2}\ldots A_{w _n}$ as concatenation of strings in $\{A_0,A_1,\ldots,A_k\}$. The median $m_{u,v,w}$ in $Q_{kn}$ of this triplet is obtained using the majority rule, which is compatible with concatenation. Thus $m_{u,v,w}=m_1m_2\ldots m_n $ where $m_i$ is the median in $Q_k$ of the triple $(A_{u_i},A_{v_i},A_{w_i})$. It follows directly from the definition of the median in hypercubes that this median is, for example, $A_{u_i}$  if $u_i=v_i$, and $A_0$ if $A_{u_i},A_{v_i},A_{w_i}$ are all different. Therefore $m_{u,v,w}$ belongs to  the monoid generated by $\A=\{A_0,A_1,\ldots,A_k\}$.
Furthermore if $u_i=v_i=k$ then, since $u$ and $v$ are Munarini strings, $u_{i+1}=v_{i+1}=0$ therefore $m_{i+1}=A_0$. In conclusion  $m_{u,v,w}\in \Fib_{n,k}^\star$,  and $\Gamma_{n,k}^\star$ is a median closed subgraph of $Q_{kn}$ therefore it is a median graph.
\end{proof}
\qed

Fibonacci cubes and many variations of them like Lucas cubes, alternate Lucas-cube~\cite{ESS-2021e} and  Pell graph~\cite[Theorem 9.68]{EKM-2023} are examples of daisy cubes.

As Taranenko noted in~\cite[Corollary 2.3]{T-2020a}, if $G$ is a daisy cube then the vertex labeled $0^m$ in a proper embedding in $Q_m$ is a vertex of degree $\Delta(G)$. Furthermore, by  definition, a vertex of degree 1 in a daisy cube must be adjacent to the vertex labeled $0^m$ in a proper embedding.  It is therefore immediate that $\Pi_{2,3}$ is not a daisy cube (see Figure~\ref{fig:Pi23M23}). More generally, for $n\geq2$ and $k\geq3$, $\Pi_{n,k}$ is not a daisy cube (See our Corollary~\ref{co:notdaisycube} in Section~\ref{sec:cub}). On the other hand we will prove  that Munarini graphs are always daisy cubes and a proper embedding is given by  $\widehat{\Psi}$.

Let $\A_{n,k}$ be the set of strings in $\Fib_{n,k}$ without $0$. For example $\A_{3,2}=\{122,221,111\}$ (see Figure~\ref{fig:isom2}).

\begin{theorem}
If $n\geq 1$ and $k\geq1$, then $M_{n,k}$ is a daisy cube. Furthermore for $k\geq2$ $\widehat{\Psi}$ is a proper embedding of $M_{n,k}$ in $Q_{kn}$ and $\Gamma_{n,k}^\star$ is the daisy cube generated by  $$X=\{\widehat{\Psi}(u)\st u\in \A_{n,k}\}\,.$$ The maximal vertices of this daisy cube are the strings of $X$, thus the images by $\widehat{\Psi}$ of the vertices of $M_{n,k}$ without $0$s.
\end{theorem}
\begin{proof}
Since $\Gamma_{n-1}$ is a daisy cube we can assume $k\geq2$. Let $K=\{v\in B^{kn}\st v\le x\ {\rm for\ some}\ x\in X \}$ and let $u$ be a string in $\Fib_{n,k}$. If there exists an occurrence 
of $0$ in the $(k+1)$-ary string $u$, then $u=x0y$ where $x$ and $y$ are 
generalized Pell strings.  Note that $u_1=x1y$ also belongs to $\Fib_{n,k}$ since $k\geq2$. 
Then $\widehat{\Psi}(u)=\widehat{\Psi}(x)0^k\widehat{\Psi}(y)$ and $\widehat{\Psi}(u_1)=\widehat{\Psi}(x)10^{k-1}\widehat{\Psi}(y)$. Thus we obtain a string with one fewer occurrence of $0$ with $\widehat{\Psi}(u)\leq\widehat{\Psi}(u_{1})$.

Repeating this process, we obtain a generalized Pell string $u_m$ without $0$, therefore in $\A_{n,k}$, such that 
$\widehat{\Psi}(u)\leq\widehat{\Psi}(u_m)$ consequently $\widehat{\Psi}(u) \in K$.  Since any string in $\Fib_{n,k}^*$ is the image of some string in $\Fib_{n,k}$ we deduce $\Fib_{n,k}^*\subset K$. 

Now consider $v \in K$ and let $x\in X$ be such that  $v\leq x$. Since $x=\widehat{\Psi}(u)$ for some $u\in \A_{n,k}$, $x$ belongs to the monoid generated by $\{A_1,A_2,\ldots,A_{k-1},A_kA_0\}$. For $i\in\{1,\ldots,k-1\}$ replacing a $1$ with a $0$  in $A_i$ gives $A_0$. Similarly changing a $1$ to $0$  in $A_kA_0$ gives $A_0A_0$. Since $v\leq x$, $v$ is obtained by replacing some occurrences of $1$ with $0$ in $x$. Therefore  $v$ is also a
string over $\{A_1,A_2,\ldots,A_{k-1},A_kA_0\}$, thus $v$ is a Munarini string. This establishes the reverse inclusion.
 
Since $\Fib_{n,k}^*=K$, the subgraph $\Gamma_{n,k}^\star$ of $Q_{kn}$ is the daisy cube generated by $X$. Since the daisy cube generated by some $X$ is a partial cube~\cite[Proposition 2.1]{KM-2019a} and $\widehat{\Psi}$ is a graph isomorphism, $\widehat{\Psi}$ is a proper embedding in $Q_{kn}$ and therefore $M_{n,k}$ is a daisy cube.

We have to prove that the elements of $X=\widehat{\Psi}(\A_{n,k})$ are maximal. 
Assume $x\leq y$ where $x \in \widehat{\Psi}(\A_{n,k})$ and $y \in K$. 
Let $z\in \widehat{\Psi}(\A_{n,k})$ such that $y \le z$. By transitivity $x \le z$. 
Since $x$ and $z$ are  words  over the alphabet $\{10^{k-1},\ldots,0^i10^{k-i-1},\ldots,0^{k-2}10,0^{k-1}10^k\}$  and $x\le z$, 
if $x$ starts $0^i10^{k-i-1}$ for some $i\in\{0,\ldots,k-1\}$, then so does $z$. 
Processing from left to right, we deduce $x=z$. 
Thus the elements of  $\widehat{\Psi}(\A_{n,k})$ are maximal and we are done.
\end {proof}
\qed

%%%%%%%%%%%%%%%%%%%%%%%%%%%%%%%%%
%%%%%%%%%%%%%%%%%%%%%%%%%%%%%%%%%%%%%%%%%%%%%%%%%%%%%%%%%%

\section{{Hypercubes as subgraphs of Munarini graphs}}
\label{sec:cub}
%%%%%%%%%%%%%%%%%%%%%%%%%%%%%%%%%%%%%%%%%%%%%%%%%%%%%%%%%%
%%%%%%%%%%%%%%%%%%%%%%%%%%%%%%%%%%%%%%%%%%%%%%%%%%%%%%%%%%

Since Munarini graphs are daisy cubes, their cube polynomial and maximal cube polynomial can be easily determined once the weight polynomial is known. This is the main purpose of this section.

For a graph $G$, let $c_p(G)$ $(p\ge 0)$ denote the number of induced 
subgraphs of $G$ isomorphic to  $Q_p$. The {\em cube polynomial} $C_G(x)$
of $G$  is the corresponding enumerator polynomial, that is,

\begin{equation}\label{eqn:defCG}
C_G(x) = \sum_{p\geq 0} c_p(G) x^p\,.
\end{equation}
This polynomial was introduced in~\cite{BKS-2003} and determined for Fibonacci and Lucas cubes in~\cite{KM-2012a} and afterwards for several of their variations~(see \cite[Chapter~9]{EKM-2023} for example and \cite{M-2025b}). 

Let $G$ be a subgraph of $Q_n$. A bivariate refinement of $C_{G}(x)$ is the {\em distance cube polynomial} (with respect to $0^n$)~\cite{KM-2019a}. This polynomial keeps track of the distance of the hypercubes to $0^n$.  We define the polynomial as follows

\begin{equation}\label{eqn:defCGq}
D_G(x,q) = \sum_{p\geq 0} c_{p,d}(G) x^pq^d 
\end{equation} where $c_{p,d}(G)$ is  the number of induced 
subgraphs of $G$ isomorphic to $Q_p$ whose bottom vertex is at distance $d$ from $0^n$.

A {\em maximal hypercube}  of a graph $G$ is an induced subgraph $H$, isomorphic to a hypercube, such that $H$ is not contained in a larger induced hypercube of $G$. For a given interconnection topology, it is important to characterize maximal hypercubes, for example from the point of view of embeddings.

Let $h_p(G)$ be the number of maximal hypercubes of dimension $p$ of $G$ and $H_{G}(x)$ the corresponding 
enumerator polynomial, that is, 

$$H_{G}(x) = \sum_{p\geq 0} h_p(G) x^p\,.$$

This polynomial was determined for Fibonacci and Lucas cubes~\cite{M-2012a} and Pell graphs~\cite{EKM-2023}.

The \emph{weight} of a binary string of length $n$ is the number of occurrences of  $1$; it is thus the Hamming distance to the string $0^n$. Since $M_{n,k}$ is isometrically embedded by $\widehat{\Psi}$ in  the  subgraph $\Gamma_{n,k}^*$ of $Q^{kn}$ we call, by extension, the weight $w(u)$ of a vertex of $M_{n,k}$ the weight of $\widehat{\Psi}(u)$. Since $\widehat{\Psi}(0^n)=0^{kn}$ we have
$$w(u)=w(\widehat{\Psi}(u))=d_{\Gamma_{n,k}^*}(\widehat{\Psi}(u),0^{kn})=d_{M_{n,k}}(u,0^n)\,.$$

Note that the weight of a vertex $u$ of $M_{n,k}$ can be expressed as $w(u)=\sum_{i=1}^{k-1}{|u|_i}+\frac{1}{2}|u|_{k}$.

The \emph{weight enumerator polynomial of a daisy cube} $G$ properly embedded in $Q^m$ is the counting polynomial $W_{G}(x)=\sum_{u\in V(G)}{x^{w(u)}}$ of the number of vertices of $G$ of a given weight. The weight enumerator polynomial of ${M_{n,k}}$ is therefore  
$$W_{M_{n,k}}(x)=\sum_{u\in V(M_{n,k})}{x^{w(u)}}=W_{\Gamma_{n,k}^*}(x)\,.$$

\begin{theorem}\label{th:geneW}
Let $k\geq1$. Then the generating function of $ W_{M_{n,k}}(x)$ is 
$$ \sum_{n\geq0}{ W_{M_{n,k}}(x)t^n}=\frac{1}{1-t-(k-1)x t -xt^2}\,.$$
\end{theorem}
\begin{proof}
Let $n\geq 2$. From the recursive decomposition of $M_{n,k}$ (see Figure~\ref{fig:Recursive}), since $0^n$ belongs to $00M_{n-2}$, we deduce 
$$  W_{M_{n,k}}(x)=W_{M_{n-1,k}}(x)+(k-1) x W_{M_{n-1,k}}+x W_{M_{n-2,k}} \,.$$
Furthermore $W_{M_{0,k}}=1$ and  $W_{M_{1,k}}=1+(k-1)x$. The result from these relations.
\end{proof}
\qed

\begin{corollary}\label{co:ndbweight}
Let $k\geq 1$ and $0\leq d \leq n$, then the number of vertices of $M_{n,k}$ at distance $d$ from $0^n$ is 
\begin{equation}\label{eqn:ndbweight}
 \sum_{j= 0}^d{\binom{d}{j}\binom{n-j}{d}(k-1)^{d-j}}
\end{equation}
and therefore 
\begin{equation}\label{eqn:polweight}
 W_{M_{n,k}}(x)=\sum_{d=0}^{n}{\sum_{j= 0}^{min (d,n-d)}{\binom{d}{j}\binom{n-j}{d}(k-1)^{d-j}}x^d}\,.
\end{equation}
\end{corollary}
\begin{proof}
\begin{align*}
\sum_{n\geq0}{ W_{M_{n,k}}(x)t^n}&=\frac{1}{1-t-(k-1)x t -xt^2}=\sum_{a\geq0}{(t+x t (k-1+t))^a}\\
       & = \sum_{a\geq0}{\sum_{b=0}^{a}{\binom{a}{b} (k-1+t)^b t^a x^b}}\\
			& = \sum_{a\geq0}{\sum_{b=0}^{a}{\sum_{j=0}^{b}{\binom{a}{b}\binom{b}{j} (k-1)^{b-j} t^{a+j} x^b}}}\,.
\end{align*}
A summand contributes to $x^dt^n$ if and only if $b=d$ and $a=n-j$.

Note that the sum can be restricted to $j\leq min(d,n-d)$ and thus (\ref{eqn:polweight}) follows.

\end{proof}
\qed
 
For $k=1$ since $0^{d-j}=0$ except for $d=j$, equality (\ref{eqn:polweight}) reduces to $W_{M_{n,1}}(x)=\sum_{d=0}^{n}\binom{n-d}{d}x^d$ the result known for $\Gamma_{n-1}$.

Corollary~\ref{co:ndbweight} also admits a combinatorial proof. Consider a vertex $v$ of $M_{n,k}$ with weight $d$ containing $j\leq d$ occurrences of $kk$. Replace each $kk$ with the symbol $\star$.  We obtain, in a bijective way, a word $\widehat{v}$ of length $n-j$ on the alphabet $\{0,1,\ldots,k-1,\star\}$ with $d$ characters belonging to $\{1,\ldots,k-1,\star\}$. After removing the $n-j-d$ $0$s we obtain a word of length $d$  over the alphabet $\{1,\ldots,k-1,\star\}$ with $j$ occurrences of $\star$. There exist $\binom{d}{j}(k-1)^{d-j}$ words of this type.
Conversely, from such a word $w$, there are $\binom{n-j}{n-j-d}=\binom{n-j}{d}$ ways to insert ${n-j-d}$ $0$s and thus  $w$ can be obtained from $\binom{n-j}{d}$ different $\widehat{v}$, themselves images of  different vertices $v$.

The following result is proved in ~\cite{KM-2019a}.

\begin{theorem}
\label{th:DfromCandCfromW}
If $G$ is a daisy cube, then $D_{G,0^n} (x,y)= C_{G}(x+y-1)$, $C_{G}(x)=W_{G}(x+1)$ and thus $ C_{G}(-1)=1$.  
\end{theorem}

Since Munarini graphs are daisy cubes, replacing $x$ by $x + 1$ in the generating function of the weight polynomial obtained in
Theorem~\ref{th:geneW} gives the generating function of the cube polynomial.

\begin{theorem}\label{th:geneC}
Let $k\geq1$, then the generating function of $C_{M_{n,k}}(x)$ is 
$$ \sum_{n\geq0}{ C_{M_{n,k}}(x)t^n}=\frac{1}{1-kt-(k-1)x t -(1+x)t^2}\,.$$
\end{theorem}
 
Note that the generating functions of $C_{M_{n,k}}(x)$ and $C_{\Pi_{n,k}}(x)$ (\cite[Proposition~5.2]{IKT-2023}))
are the same. Therefore, not only do $M_{n,k}$ and $\Pi_{n,k}$ have the same order and size, but they also have the same number of induced $Q_p$ for any $p$.

\begin{corollary}\label{co:notdaisycube}

Let $k\geq3$ and $n\geq2$, then $\Pi_{n,k}$ is not a daisy cube.
\end{corollary}
\begin{proof}
If $G$ is a daisy cube and $u$ is the vertex labeled $0^m$ in a proper embedding, then the coefficient of $x$ in the weight polynomial $W_G(x)$ is the number of neighbors of $u$, hence its degree. By  
~\cite[Corollary 2.3]{T-2020a} this degree is necessarily $\Delta(G)$. Assume that $\Pi_{n,k}$ is a daisy cube. Then, since $\Pi_{n,k}$ and $M_{n,k}$ have the same cube polynomial, by Theorem~\ref{th:DfromCandCfromW}, they necessarily have the same weight polynomial, hence they must have the same maximum degree.  It is proved in \cite[Corollary~5.5]{IKT-2023} that for $k\geq3$, $\Delta(\Pi_{n,k})=2n$. Since, in a string of $\Fib_{n,k}$, each $0$ can be changed to some $i\in\{1,\ldots,k-1\}$ and $00$ to $kk$ it is immediate that the degree of $0^n$ in $M_{n,k}$ is  $(k-1)n+n-1=kn-1$ for $n \geq1$ and  we obtain a contradiction. 
\end{proof}
\qed

A consequence of Corollary~\ref{co:ndbweight} and Theorem~\ref{th:DfromCandCfromW} is an expression of the coefficients of the cube polynomial, which is a generalization of the formula known for Pell graphs~\cite[Equation (14)]{M-2019}.

\begin{corollary}\label{co:ndQp}
Let $k\geq 1$, $n\geq 0$  and $p\leq n$, then the number of induced subgraphs  of $M_{n,k}$ isomorphic to $Q_p$ is
\begin{equation*}
 c_p(M_{n,k})=\sum_{d=p}^n{\sum_{j=0}^d{\binom{d}{p}{\binom{d}{j}\binom{n-j}{d}(k-1)^{d-j}}}}\,.
\end{equation*}
\end{corollary}
\begin{proof}
Replacing $x$ by $x+1$ in equation~(\ref{eqn:polweight}) we obtain the following expression of the cube polynomial
\begin{align*}\label{eqn:pol}
 C_{M_{n,k}}(x)&=\sum_{d=0}^{n}{\sum_{j= 0}^{d}{\binom{d}{j}\binom{n-j}{d}(k-1)^{d-j}}(1+x)^d}\\
&=\sum_{d=0}^{n}{\sum_{j= 0}^{d}{\binom{d}{j}\binom{n-j}{d}(k-1)^{d-j}}\sum_{p=0}^{d}{\binom{d}{p}}x^p}\,.
\end{align*}
Since, for $d<p$, $\binom{d}{p}=0$ the result follows. 

\end{proof}\qed

For $k=1$ then again $j=d$ and the expression reduces to $c_p(M_{n,1})=\sum_{d=p}^n\binom{d}{p}\binom{n-d}{d}$, the result for $\Gamma_{n-1}$\cite[Corollary 3.3]{KM-2012a}.

Another consequence of Theorem~\ref{th:DfromCandCfromW} for daisy cubes, and therefore  for Munarini graphs, is that the number of edges is the sum of the weights.
\begin{proposition}
If $G$ is a daisy cube with weight polynomial $W_G(x)$ then $|E(G)|= |W(G)|$ where $W(G)$ is the total weight number $W(G)=\sum_{u\in V(G)}w(u)=W'_G(1)$.
\end{proposition}
\begin{proof}
$|E(G)|$ is the coefficient of $x$ in $Q_G(x)=W_G(x+1)$. This coefficient is $Q'_G(0)$. By derivation of a composite function $Q'_G(x)=W'_G(x+1)$ and the result follows.
Note that there is a combinatorial interpretation of this result. Indeed by counting the edges by their top vertex, a vertex of weight $d$, therefore at distance $d$ of $0^m$,  is the top vertex of exactly $d$ edges in $Q^m$ thus in $G$.
\end{proof}\qed

Another parameter of Pell graphs considered by Munarini is the cube number.
The \emph{cube number } $q(G)$ of a graph $G$ is the number of hypercubes in $G$. Therefore   $q(G)=\sum_{p\geq0}c_p(G)=C_G(1)$. Munarini proved \cite[Proposition~22]{M-2019} that the generating function of the sequence $(q(\Pi_n))_{n\geq 0}$ is $\sum_{n\geq0}{q(\Pi_n)t^n}=\frac{1}{1-3t-2t^2}$. Setting  $x=2$ in the generating function of $W_{M_{n,k}}(x)$ we deduce the generalization to Munarini graphs and generalized Pell graphs.
\begin{proposition}
Let $k\geq1$, then the generating function of the sequence $(q(M_{n,k}))_{n \geq 0}$ and $(q(\Pi_{n,k}))_{n \geq 0}$ are
$$ \sum_{n\geq0}{q(M_{n,k})t^n}=\sum_{n\geq0}{q(\Pi_{n,k})t^n}=\frac{1}{1-(2k-1)t -2t^2}\,.$$
\end{proposition}

For $k=1$ we obtain that $q(\Gamma_n)=J_{n+2}$ where  $(J_n)_{n\geq0}$ is the Jacobsthal sequence~\cite[A001045]{Sloane}. The cube numbers of $\Pi_n$, $M_{n,3}$ and $M_{n,4}$ also appear in OEIS \cite{Sloane} under the references A007482, A015535 and A015555. This suggests the following combinatorial interpretation.

\begin{proposition} 
For $k\geq 1$, let $A_{n,k}$ be the set of words of length $n$ over the alphabet $\{0,1,\ldots,2k\}$ without runs of $0$s and $1$s of odd lengths.  Then the cube number of $\M_{n,k}$ is $|A_{n,k}|$.
\end{proposition}
 \begin{proof}
Let $a_{n,k}=|A_{n,k}|$. For $n\geq 2$, a word in $A_{n,k}$ is obtained either by concatenating  a symbol $i\in \{2,\ldots,2k\}$  with a word in $A_{n-1,k}$ or by concatenating  $00$  or  $11$ with a word in $A_{n-2,k}$. Therefore 
$a_{n,k}=(2k-1)a_{n-1,k}+2a_{n-2,k}$.  From the generating function  $q(M_{n,k})$ follows the same relation. Since $a_{0,k}=1=q(M_{0,k})$,  $a_{1,k}=2k-1$ and $q(M_{1,k})=q(S_{k-1})=2k-1$ the sequences $(a_{nk})_{n\geq 0}$ and $(q(M_{n,k}))_{n \geq 0}$ are the same.

There exists a bijective proof. Consider the same order alphabet $\Sigma=\{0,1,\ldots,k\}\cup\{1',2',\ldots,k'\}$.  An induced hypercube  $H$ of $M_{n,k}$ is 
characterized by its top vertex $v$ and its support. Let $$I=\{i\in\{1,\ldots ,n\} \st v_i\in\{1,\ldots, k-1\}\}\cup \{i\in\{1,\ldots ,n-1\} \st v_iv_{i+1}=kk \}.$$ We can associate to the support of $H$ a subset $J$ of $I$. Since  $M_{n,k}$ is a daisy cube, to every subset of $I$ the corresponding hypercube with top vertex $v$ is a subgraph of $M_{n,k}$ by Remark~\ref{re:daisy}. Let's replace $v_i$ with $v_i'$ if $i\in J$ and $v_i\in\{1,\ldots,k-1\}$, and  $v_iv_{i+1}$ with $v_i'v_{i+1}'$ if $i\in J$ and $v_i=k$. This yields a word of length $n$ over the alphabet $\Sigma$ with no runs of $k$ or $k'$  of odd length. The set of these words is in bijection with $A_{n,k}$. The inverse operation is possible: from such a word, removing the primes give the top vertex and the positions of the primes define the set $J$ therefore the support. Since $M_{n,k}$ is a daisy cube, the corresponding hypercube is a subgraph of $M_{n,k}$. 
\end{proof}
\qed

Maximal hypercubes are easy to determine for daisy cubes.

\begin{proposition}\cite[Proposition 9.69]{EKM-2023}
\label{pro:daisymax}
Let $Q_n(X)$ be a daisy cube. The number of maximal hypercubes of dimension $k$ in $Q_n(X)$ is the number of maximal vertices of Hamming weight $k$. 
\end{proposition}

Let $k\geq2$. Maximal vertices in the daisy cube $M_{n,k}$ are the image by the proper  embedding $\widehat{\Psi}$ of strings of $\Fib_{n,k}$ without $0$s.  Let us recall that, for any vertex $u$ of $M_{n,k}$, we have the equality $d_{M_{n,k}}( u,0^n)=d (\widehat{\Psi}(u),0^{kn})=w(\widehat{\Psi}(u)) $. Therefore the number $h_p(M_{n,k})$ of maximal hypercubes of dimension $p$ is  the number of strings of $\Fib_{n,k}$ without $0$ at distance $p$ in $M_{n,k}$ of $0^n$.
For $n\geq2$ and $p\geq1$, from the decomposition of $M_{n,k}$, see Figure~\ref{fig:Recursive},  we infer that $h_p(M_{n,k})= (k-1) h_{p-1}(M_{n-1,k})+h_{p-1}(M_{n-2,k})$ because $0^n$ belongs to $00M_{n-2,k}$ and a string without $0$s cannot be in $0M_{n-1,k}$.

Therefore the family of polynomials  $H_{M_{n,k}}(x) = \sum_{p\geq 0} h_p(M_{n,k}) x^p$ satisfies for $n\geq 2$
\begin{equation}\label{eq:max}H_{M_{n,k}}(x)=x(k-1) H_{M_{n-1,k}}(x) + x H_{M_{n-2,k}}(x)\,.\end{equation}

Together with $H_{0,k}$=1 and $H_{1,k}=(k-1)x$ we deduce the following result.

\begin{theorem}\label{th:geneH}
Let $k\geq2$, then the generating function of the maximal cube polynomial $ H_{M_{n,k}}(x)$ is 
$$ \sum_{n\geq0}{ H_{M_{n,k}}(x)t^n}=\frac{1}{1-(k-1)x t -xt^2}\,.$$
\end{theorem}
\begin{corollary}\label{co:geneH}
Let $k\geq2$, then the number of maximal hypercubes of dimension $p$ in $M_{n,k}$ is
$$
h_p(M_{n,k})= (k-1)^{2p-n}\binom{p}{n-p}\,.
$$
\end{corollary}
\begin{proof}
Expanding the generating function, we obtain
$$\frac{1}{1-(k-1)x t -xt^2}=\sum_{a\geq0}{((k-1)x t+xt^2)^a}=\sum_{a\geq0}{\sum_{b=0}^a{\binom{a}{b}x^a t^{a+b}(k-1)^{a-b}}} \,.$$
A term in the sum contributes to $x^pt^n$ if and only if $a=p$ and $b=n-p$. The result follows.

We can give a combinatorial interpretation. The number of maximal hypercubes of dimension $p$ in $M_{n,k}$ is the number of vertices $u$ in $M_{n,k}$ at a distance $p$ of $0^n$ such that $u$ does not use the letter $0$. Such a string uses $a$ letters in $\{1,\ldots,k-1\}$ and $b$ letters $kk$  with $a+b=p$ and $a+2b=n$.
Therefore, $a=2p-n$ and $b=n-p$.  There exist $(k-1)^{a}$ words of length $a$ over $\{1,\ldots,k-1\}$. Given this word, there are $\binom{a+b}{b}$ ways to insert the $b$ letters $kk$ into this word.
\end{proof}
\qed

Note that $h_p(M_{n,k})=0$ for $p<\frac{n}{2}$. 

Maximal vertices in Fibonacci cubes have a different construction; thus, the expression of $h_p(M_{n,1})$ is different, $h_p(M_{n,1})=h_p(\Gamma_{n-1})=\binom{p+1}{n-2p}$~ \cite[Proposition~2.5]{M-2012a}. However Corollary~\ref{co:geneH} generalizes the expression known for Pell graphs~\cite[Theorem~9.70]{EKM-2023}.

The generating function of $ H_{M_{n,k}}(x)$ can be determined alternatively, using the recursive decomposition (see Figure~\ref{fig:Recursive}, right).  Indeed a maximal hypercube in $M_{n,k}$ can be
\begin{itemize}
\item a maximal hypercube in one of the $k-1$ subgraphs induced by $i\Fib_{n-1,k}\dot{\cup} 0\Fib_{n-1,k}$ for some $i\in\{1,\ldots,k-1\}$,
\item a maximal hypercube in the  subgraph induced by $kk\Fib_{n-2,k}\dot{\cup} 00\Fib_{n-2,k}$.
\end{itemize}
Since these subgraphs are respectively isomorphic to $M_{n-1,k}\cp K_1$ and $M_{n-2,k}\cp K_1$, we deduce that, for $n\geq 2 $ and $p \geq 1$,
\begin{equation}\label{eq:hp}
 h_p(M_{n,k})= (k-1)h_{p-1}(M_{n-1,k})+ h_{p-1}(M_{n-2,k})\,.
\end{equation}
The expression of the generating  function of $ H_{M_{n,k}}(x)$ is then immediate from  the initial values $H_{M_{1,k}}=H_{S_{k-1}}=(k-1)x$ and $H_{M_{0,k}}=1$.

The same idea can be used to determine  the generating function of the maximal cube polynomial  of generalized  Pell graphs, using their recursive decomposition (see Figure~\ref{fig:Recursive}, left). A maximum hypercube in $\Pi_{n,k}$ can be
\begin{itemize}
\item a maximum hypercube in one of the $k-1$ subgraphs induced by $i\Fib_{n-1,k}\dot{\cup}(i-1)\Fib_{n-1,k}$ for some $i\in\{1,\ldots,k-1\}$,
\item a maximum hypercube in the  subgraph induced by $kk\Fib_{n-2,k}\dot{\cup} (k-1)(k-1)\Fib_{n-2,k}$.
\end{itemize} 
Therefore, $h_p(\Pi_{n,k})$ follows a similar relation
\begin{equation}\label{eq:hpM}
 h_p(\Pi_{n,k})= (k-1)h_{p-1}(\Pi_{n-1,k})+ h_{p-1}(\Pi_{n-2,k})\,.
\end{equation}
Since $H_{\Pi_{1,k}}=H_{P_{k}}= (k-1)x= H_{M_{1,k}}$ we deduce the following result.
\begin{proposition}
Let $n\geq0$ and  $k\geq 1$,  then  $\Pi_{n,k}$ and  $M_{n,k}$ have the same maximal cube polynomial. 
\end{proposition}

It is important to note that even though $\Pi_{n,k}$ and $M_{n,k}$ have the same maximal cube polynomial, the bijection between maximal hypercubes does not seem obvious; it is not, for example, an extension of one of the bijections between the vertices of these graphs. Indeed, in the daisy cube $M_{2,3}$, the maximal hypercubes cannot be disjoint because they share the same bottom vertex, whereas in $\Pi_{2,3}$, the hypercubes induced by $\{11,01,00,10\}$ and $\{22,33\}$, although disjoint, are maximal (see Figure~\ref{fig:Pi23M23}, left).

\section{Conclusion}
\label{sec:asy}
In this paper, we introduce Munarini graphs, a generalization of Fibonacci cubes and Pell graphs, and study their first properties.These results show that Munarini graphs share several structural and enumerative properties with Fibonacci cubes and Pell graphs, further supporting their role as a natural generalization within the class of daisy cubes.

In a follow-up paper, we will determine the diameter, radius, and the number of vertices of a given degree or eccentricity, and conclude with some asymptotic results.

There are many other unanswered questions about Munarini graphs.  For instance, Fibonacci cubes appear as resonance graphs of a class of benzenoid hydrocarbons in mathematical chemistry~\cite{KZ-2005a}. More generally, the connection between resonance graphs and daisy cubes has also been explored; see for example~\cite{Z-2018}. Clarifying the role of Munarini graphs in this context would be interesting.
%
%%%%%%%%%%%%%%%%%%%%%%%%%%%%%%%%%%%%%%%%%%%%%%%%%%%%%%%%%%%
%%%%%%%%%%%%%%%%%%%%%%%%%%%%%%%%%%%%%%%%%%%%%%%%%%%%%%%%%%%
%
%\section{Degrees in Munarini graphs}
%\label{sec:deg}
%%%%%%%%%%%%%%%%%%%%%%%%%%%%%%%%%%%%%%%%%%%%%%%%%%%%%%%%%%%
%%%%%%%%%%%%%%%%%%%%%%%%%%%%%%
%
%%%%%%%%%%%%%%%%%%%%%%%%%%%%%%%%%%%%%%%%%%%%%%%%%%%%%%%%%%%
%%%%%%%%%%%%%%%%%%%%%%%%%%%%%%%%%%%%%%%%%%%%%%%%%%%%%%%%%%%
%
%\section{Eccentricity, diameter and radius}
%\label{sec:ecc}
%%%%%%%%%%%%%%%%%%%%%%%%%%%%%%%%%%%%%%%%%%%%%%%%%%%%%%%%%%%
%%%%%%%%%%%%%%%%%%%%%%%%%%%%%%%%%%%%%%%%%%%%%%%%%%%%%%%%%%%
%
%%%%%%%%%%%%%%%%%%%%%%%%%%%%%%%%%%%%%%%%%%%%%%%%%%%%%%%%%%%
%%%%%%%%%%%%%%%%%%%%%%%%%%%%%%%%%%%%%%%%%%%%%%%%%%%%%%%%%%%
%
%\section{Cube polynomial of Fibonacci $p$-cubes}
%\label{sec:cubepolynomial}
%%%%%%%%%%%%%%%%%%%%%%%%%%%%%%%%%%%%%%%%%%%%%%%%%%%%%%%%%%%
%%%%%%%%%%%%%%%%%%%%%%%%%%%%%%%%%%%%%%%%%%%%%%%%%%%%%%%%%%%
%
%
%%%%%%%%%%%%%%%%%%%%%%%%%%%%%%%%%%%%%%%%%%%%%%%%%%%%%%%%%%%
%%%%%%%%%%%%%%%%%%%%%%%%%%%%%%%%%%%%%%%%%%%%%%%%%%%%%%%%%%%
%
%\section{Asymptotic properties}
%\label{sec:asy}
%%%%%%%%%%%%%%%%%%%%%%%%%%%%%%%%%%%%%%%%%%%%%%%%%%%%%%%%%%%
%%%%%%%%%%%%%%%%%%%%%%%%%%%%%%%%%%%%%%%%%%%%%%%%%%%%%%%%%%%
%%
%%
%
%
\bibliographystyle{plain}  

\bibliography{fibo-bib4}

\end{document}